\documentclass[reqno]{amsart}

\usepackage{verbatim}
\usepackage{graphicx}
\usepackage{amssymb}
\usepackage{esint}

\newcommand{\R}{\mathbb{R}}

\newcommand{\grad}{\nabla}
\newcommand{\lap}{\Delta}

\theoremstyle{plain}
\newtheorem{thm}{Theorem}[section]
\newtheorem{lemma}[thm]{Lemma}
\newtheorem{defn}[thm]{Definition}

\numberwithin{equation}{section}

\newtheorem{rem}[thm]{Remark}

\begin{document}
\title[3D nematic liquid crystal flow] {Global existence of weak solutions of the nematic liquid crystal flow in dimensions three} 

\author[F. Lin]{Fanghua Lin\ \ \ } 
\address{Courant Institute of Mathematical Sciences\\
New York University\\
NY 10012 and NYU-ECNU Institute of Mathematical Sciences, at NYU Shanghai, 3663, North Zhongshan Rd., Shanghai, PRC 200062}
\email{linf@cims.nyu.edu}
\author[C. Wang]{\ \ \ Changyou Wang} 
\address{Department of
Mathematics, Purdue University, 150 N. University Street, West Lafayette, IN 47907, USA}
\email{wang2482@purdue.edu} 
\date{\today}
\subjclass[2000]{}

\begin{abstract}
For any bounded smooth domain $\Omega\subset\mathbb R^3$ (or $\Omega=\mathbb R^3$),
we establish the global existence of a weak solution $(u,d):\Omega\times [0,+\infty)\to\mathbb R^3\times\mathbb S^2$
of the initial-boundary value (or the Cauchy) problem of
the simplified Ericksen-Leslie system (\ref{LLF}) modeling the hydrodynamic flow of nematic liquid crystals
for any initial and boundary (or Cauchy) data $(u_0,d_0)\in {\bf H}\times H^1(\Omega, \mathbb S^2)$,
with $d_0(\Omega)\subset\mathbb S^2_+$ (the upper hemisphere). Furthermore, $(u,d)$ satisfies the global energy inequality (\ref{global_energy_ineq}).
\end{abstract}
\maketitle 

\section{Introduction}
\setcounter{equation}{0}
\setcounter{thm}{0}
In this paper, we consider the following simplified Ericksen-Leslie system modeling the  hydrodynamics of nematic liquid crystals
in dimensions three:
for a bounded smooth domain $\Omega\subset\mathbb R^3$ (or $\Omega=\mathbb R^3$)
and $0<T\le\infty$,  $(u,P,d):\Omega\times[0,T) \to \R^3\times\R\times \mathbb S^2$ solves
\begin{equation}\label{LLF}
\begin{cases}
\begin{aligned}
\partial_t u+ u \cdot \grad u - \nu\lap u + \grad P &= -\lambda\grad\cdot(\grad d \odot \grad d), \ \ {\rm{in}}\ \Omega\times (0,T),\\
\grad \cdot u &= 0, \ \ \ \ \ \ \ \ \ \  \ \ \ \ \ \ \ \ \ \ \ \ \ {\rm{in}}\ \Omega\times (0,T),\\
\partial_t d + u \cdot \grad d &=\gamma(\lap d+ |\grad d|^2 d), \ \ \ \ {\rm{in}}\ \Omega\times (0,T),
\end{aligned}
\end{cases}
\end{equation}
along with the initial and boundary condition:
\begin{equation}\label{IBC}
\begin{cases}
\begin{aligned}
(u, d)&=(u_0, d_0) \ \ {\rm{in}} \ \ \Omega\times \{0\},\\
(u,d)&=(0,d_0)\ \ \ {\rm{on}} \ \ \partial\Omega\times (0,+\infty),
\end{aligned}
\end{cases}
\end{equation}
for a given data $(u_0,d_0):\Omega\to\mathbb R^3\times\mathbb S^2$, with $\nabla\cdot u_0=0$.
Here $u:\Omega\to\mathbb R^3$ represents the velocity field of the fluid, $d:\Omega\to \mathbb S^2$ (the unit sphere
in $\mathbb R^3$) is a unit vector field representing the macroscopic orientation of the nematic liquid crystal molecules,
and $P:\Omega\to\R$ represents the pressure function. The constants $\nu,\lambda,$ and $\gamma$ are positive
constants representing the viscosity of the fluid, the competition between kinetic and potential energy, and the microscopic
elastic relaxation time for the molecular orientation field respectively.
$\nabla\cdot$ denotes the divergence operator in $\mathbb R^3$, and $ \grad d \odot \grad d$ denotes the
symmetric $3\times 3$ matrix: $\displaystyle
\left ( \grad d \odot \grad d \right )_{ij} = \langle\grad_id,  \grad_j d\rangle, \ 1\le i, j \le 3.$
Throughout this paper, we denote $\displaystyle\langle v, w\rangle$ or $v\cdot w$
as the inner product in $\mathbb R^3$ for $v,w\in\mathbb R^3$.

The system (\ref{LLF}) is a simplified version of the celebrated Ericksen-Leslie model for the hydrodynamics of nematic liquid crystals developed by Ericksen
and Leslie during the period of 1958 through 1968 \cite{ericksen, leslie, degenes}. The full Ericksen-Leslie system reduces to the Oseen-Frank
model of liquid crystals in the static case. It is a macroscopic continuum description of the time evolution of the materials
under the influence of fluid velocity field $u$ and the macroscopic description of the microscopic orientation field $d$ of
rod-like liquid crystals. The current form of system (\ref{LLF}) was first proposed by Lin \cite{lin} back in the late 1980's. From the mathematical point
of view, (\ref{LLF}) is a system strongly coupling the non-homogeneous incompressible Navier-Stokes equation and the
transported heat flow of harmonic maps to $\mathbb S^2$. Lin-Liu \cite{LL1, LL2} have initiated the mathematical analysis of (\ref{LLF}) by
considering its Ginzburg-Landau approximation or the so-called orientation with variable degrees in the terminology of Ericksen.
Namely, the Dirichlet energy $E(d)=\displaystyle\frac12\int |\nabla d|^2$ for $d:\mathbb R^3\to \mathbb S^2$ is replaced by  the Ginzburg-Landau
energy $E_\epsilon(d)=\displaystyle\int \frac12|\nabla d|^2+\frac{1}{4\epsilon^2}(1-|d|^2)^2$ ($\epsilon>0$) for $d:\mathbb R^3\to\mathbb R^3$.
Hence (\ref{LLF})$_3$ is replaced by
\begin{equation}\label{GL_LLF}
\partial _t d+u\cdot\nabla d= \gamma(\Delta d+\frac{1}{\epsilon^2} (1-|d|^2)d).
\end{equation}
Lin-Liu have proved in \cite{LL1, LL2} (i) the existence of a unique, global smooth solution in dimension two and in dimension three under large
viscosity $\nu$; and (ii) the existence of suitable weak solutions and their partial regularity in dimension three, analogous to the
celebrated regularity theorem by Caffarelli-Kohn-Nirenberg \cite{CKN} (see also \cite{lin1})
for the three-dimensional incompressible Navier-Stokes equation.

As already pointed out by \cite{LL1, LL2}, it is a very challenging problem to study the issue of convergence of solutions $(u_\epsilon, P_\epsilon, d_\epsilon)$ to (\ref{LLF})$_1$-(\ref{LLF})$_2$-(\ref{GL_LLF}) as $\epsilon$ tends to $0$. In particular, the existence of global Leray-Hopf type  weak solutions to the initial and boundary value problem of (\ref{LLF}) has only been established recently by Lin-Lin-Wang \cite{LLW} in dimension two, see also Hong \cite{hong} and Xu-Zhang \cite{XZ}  for related works.

Because of the super-critical nonlinear term $\nabla\cdot(\nabla d\odot\nabla d)$ in (\ref{LLF})$_1$,
it has been an outstanding open problem whether there exists a global  Leray-Hopf type weak solution to (\ref{LLF}) in $\mathbb R^3$ for any initial data $(u_0, d_0)\in L^2(\Omega,\R^3)\times {H}^{1}(\Omega, S^2)$ with $\nabla\cdot u_0=0$.
We would like to mention that  Wang \cite{wang} has recently obtained the global (or local) well-posedness of
(\ref{LLF}) for initial data $(u_0, d_0)$ belonging to possibly the largest space ${\rm{BMO}}^{-1}\times {\rm{BMO}}$ with $\nabla\cdot u_0=0$, which is an invariant space under parabolic scaling associated with (\ref{LLF}),  with small norms.

In this paper, we are interested in the global existence of weak solutions to (\ref{LLF}) for large initial data.
Since the exact values of $\nu, \lambda, \gamma$ don't play roles in this paper, we
henceforth assume
$$\nu=\lambda=\gamma=1.$$
Before stating our theorems, we need to introduce some notations.  For $b\in [-1,1]$, set
$$\mathbb S^2_{b}=\big\{y=(y^1,y^2,y^3)\in\mathbb S^2: \ y^3\ge b\big\},$$
and let $\mathbb S^2_+=\mathbb S^2_{0}$ denote the upper hemisphere. Set
$${\bf H}={\rm{Closure\ of }}\ C_0^\infty(\Omega,\mathbb R^3)\cap\big\{v: \nabla\cdot v=0\big\} \ {\rm{in}}\ L^2(\Omega,\mathbb R^3),$$
$${\bf J}={\rm{Closure\ of }}\ C_0^\infty(\Omega,\mathbb R^3)\cap\big\{v: \nabla\cdot v=0\big\} \ {\rm{in}}\ H^1_0(\Omega,\mathbb R^3),$$
and
$$H^1(\Omega,\mathbb S^2)=\big\{d\in H^1(\Omega,\mathbb R^3): \ d(x)\in\mathbb S^2 \ {\rm{a.e.}}\ x\in\Omega\big\}.
$$

In this context, we are able to prove

\begin{thm} \label{existence} For any $u_0\in {\bf H}$
and $d_0\in H^1(\Omega,\mathbb S^2)$ with
$d_0(\Omega)\subset\mathbb S^2_+$,  there exists a global weak solution $(u,d):\Omega\times [0,+\infty)
\to\mathbb R^3\times\mathbb S^2$ to the initial and boundary value problem of (\ref{LLF}) and (\ref{IBC}) such that
\begin{itemize}
\item[(i)] $u\in L^\infty_tL^2_x\cap L^2_tH^1_x(\Omega\times [0,+\infty),\mathbb R^3)$.
\item[(ii)] $d\in L^\infty_tH^1_x(\Omega, \mathbb S^2)$ and $d^3(x,t)\ge 0$ a.e.  $(x,t)\in\Omega\times (0,+\infty)$.
\item[(iii)] $(u,d)$ satisfies the global energy inequality: for $L^1$-a.e. $0\le t<+\infty$,
\begin{equation}\label{global_energy_ineq}
\int_\Omega (|u|^2+|\nabla d|^2)(t)+2\int_{0}^{t}\int_\Omega (|\nabla u|^2+|\Delta d+|\nabla d|^2d|^2)
\le \int_\Omega (|u_0|^2+|\nabla d_0|^2).
\end{equation}
\end{itemize}
\end{thm}
\begin{rem} {\rm From the proof of Theorem \ref{existence}, it is clear that the weak solution $(u,d)$
 obtained in Theorem \ref{existence} enjoys the property that for $L^1$ a.e. $t\in (0,+\infty)$,
$d(t)\in H^1(\Omega, \mathbb S^2)$ is a {\rm suitable} approximated harmonic map to $\mathbb S^2$ with tension field
$\tau(t)=(\partial_t d+u\cdot\nabla d)(t)\in L^2(\Omega, \mathbb R^3)$ (see the definition \ref{suitable}).}
\end{rem}
Based on Theorem \ref{precomp1} and Theorem \ref{precomp2}, we also establish the following
compactness property for a class of weak solutions to (\ref{LLF}) that contains
those solutions constructed by Theorem \ref{existence}.

\begin{thm}\label{compactness0} For any $0<a\le 2$ and $0<T\le +\infty$, assume that
$(u_k, d_k):\Omega\times (0,T]\to\mathbb R^3\times \mathbb S^2_{-1+a}$ is a sequence
of weak solutions of (\ref{LLF}), that satisfies
\begin{equation}\label{energy_bound}
\sup_{k\ge 1}\Big[\sup_{0\le t\le T}\int_\Omega (|u_k|^2+|\nabla d_k|^2)+
\int_{0}^T\int_\Omega(|\nabla u_k|^2+|\Delta d_k+|\nabla d_k|^2 d_k|^2)\Big]
<+\infty,
\end{equation}
and for $L^1$ a.e. $t\in (0,+\infty)$, $d_k(t)\in H^1(\Omega, \mathbb S^2)$
is a {\rm{suitable}} approximated harmonic map with tension field $\tau_k(t)=(\partial_t d_k+u_k\cdot\nabla d_k)(t)\in
L^2(\Omega,\mathbb R^3)$.
Then there exists a weak solution $(u, d):\Omega\times (0,T]\to\mathbb R^3\times \mathbb S^2_{-1+a}$ of (\ref{LLF})
such that, after passing to possible subsequences,
\begin{equation}\label{compact}
u_k\rightarrow u, \nabla d_k\rightarrow\nabla d \ \ {\rm{in}}\ \ L^2_{\rm{loc}}(\Omega\times [0,T]).
\end{equation}

\end{thm}

\medskip
The proof of Theorem \ref{existence} is very delicate. The weak solution $(u, d)$ to (\ref{LLF})  is obtained as a weak limit
of a sequence of weak solutions $(u_\epsilon, d_\epsilon)$ to the Ginzburg-Landau approximated equation of (\ref{LLF}) 
(i.e., the equations (\ref{LLF})$_1$, (\ref{LLF})$_2$, and (\ref{GL_LLF}) as $\epsilon$ tends to zero. The key ingredient
is to show that $\nabla d_\epsilon$ subsequentially converges to $\nabla d$ in $L^2_{\rm{loc}}(\Omega\times (0,+\infty))$,
or equivalently the subsequential $L^2_{\rm{loc}}$-compacteness of $\nabla d_\epsilon$. This is achieved by showing

\noindent (i) $d_\epsilon^3\ge 0$ via the maximum principle, \\
(ii) $d_\epsilon(t) $ enjoys slice almost energy monotonicity property for $L^1$-a.e.
$t>0$,\\
(iii) at good time slices $t>0$, $d_\epsilon(t)$ enjoys both regularity estimate and $H^1$-precompactness property under the small energy condition, and\\
(iv) utilizing the range assumption of $d_\epsilon$ to rule out
the defect measures generated during the blow-up analysis of $d_\epsilon(t)$ for points at good time slices
where the small energy condition may not hold. As a consequence, we actually show that at any good time slice $t$,
the small energy condition holds everywhere.  

It is in step (iv) that we need to adapt and extend the blow-up techniques the authors have developed for the heat flow
of harmonic maps in \cite{LW1, LW2, LW3}.

\begin{rem} {\rm For general initial data $d_0\in H^1(\Omega, \mathbb S^2)$ (i.e., without the assumption  $d_0^3(x)\ge 0$ for a.e. $x\in\Omega$), our blow-up analysis scheme in this paper seems to suggest that defect measures $\nu$ may result during the
convergence procedure of ($u_\epsilon, d_\epsilon$) to ($u,d$) as $\epsilon\rightarrow 0$. The defect measure 
$\nu$ represents a transported
version of curvature motion of generalized curves, and ($u,d$) is a weak solution of the nematic liquid crystal flow (\ref{LLF})
away from the support of $\nu$, which is the energy concentration set of the convergence. This energy concentration set may 
correspond to dark threads that appear in the study of liquid crystal flows.
We believe that, motivated by earlier results on the heat flow of harmonic maps \cite{LW1,LW2, LW3},
($u,d$) and $\nu$ is a weak solution of the nematic liquid crystal flow (\ref{LLF}) coupled with transported versions of
generalized $1$-varifold flows in Brakke's sense. We plan to investigate these issues in a forthcoming paper.}

\end{rem}
The paper is written as follows. In section 2, we will establish some preliminary estimates of (\ref{GL_LLF}) by the weak maximum principle.
In section 3, we will establish a slice almost monotonicity inequality of
(\ref{GL_LLF}).
In section 4, we will prove an $\delta_0$-compactness property for weak solutions to (\ref{GL_LLF}). In section 5, we will establish
an $\delta_0$-regularity for suitable approximated harmonic map to
$\mathbb S^2$. In section 6, we will establish $H^1$-precompactness
for certain solutions of (\ref{GL_LLF}). In section 7,  we will establish $H^1$-precompactness for suitable approximated harmonic maps to
$\mathbb S^2_{-1+a}$. In section 8, we will prove both Theorem
\ref{existence} and Theorem \ref{compactness0}.

\medskip

\section{Maximum principle on the transported Ginzburg-Landau heat flow}
In this section, we will establish two pointwise estimates for the transported Ginzburg-Landau
heat flow by the weak maximum principle.

For $\epsilon>0$, consider the initial-boundary value problem of the transported Ginzburg-Landau heat flow:
\begin{equation}\label{TGL}
\begin{cases}
\begin{aligned}
\partial_td_\epsilon+u_\epsilon\cdot\nabla d_\epsilon&=\Delta d_\epsilon+\frac{1}{\epsilon^2}(1-|d_\epsilon|^2) d_\epsilon
\ \ {\rm{in}}\ \ \Omega\times (0,+\infty),\\
\nabla \cdot u_\epsilon &=0 \ \qquad\qquad\qquad\qquad\ \ \ \ {\rm{in}}\ \ \Omega\times (0,+\infty),\\
d_\epsilon&=g_\epsilon  \qquad\qquad\qquad\qquad \ \ \ \ {\rm{on}}\ (\Omega\times \{0\})\cup(\partial\Omega\times (0,+\infty)).
\end{aligned}
\end{cases}
\end{equation}
\begin{lemma}\label{MP1} For $0<T<+\infty$, assume $u_\epsilon\in L^2([0,T], {\bf J})$ and
$g_\epsilon\in H^1(\Omega,\mathbb R^3)$ satisfies
$$|g_\epsilon(x)|\le 1, \ {\rm{a.e.}}\ x\in\Omega.$$
Suppose $d_\epsilon\in L^2([0,T], H^1(\Omega,\mathbb R^3))$, with
$(1-|d_\epsilon|^2)\in L^2(\Omega\times [0,T])$,  solves
(\ref{TGL}). Then $|d_\epsilon(x,t)|\le 1$ for a.e. $(x,t)\in\Omega\times [0,T]$.
\end{lemma}
\begin{proof} For any $k>1$, define $v_\epsilon^k:\Omega\times [0,T]\to \mathbb R_+$ by letting
$$v_\epsilon^k(x,t)=\begin{cases}
k^2-1& \ {\rm{if}}\ |d_\epsilon(x,t)|>k,\\
 |d_\epsilon(x,t)|^2-1 & \ {\rm{if}}
\ 1<|d_\epsilon(x,t)|\le k,\\
0 & \ {\rm{if}}\ |d_\epsilon(x,t)|\le 1.
\end{cases}
$$
Then direct calculations imply that $v_\epsilon^k$ satisfies
\begin{eqnarray}\label{2.2}
\partial_t v_\epsilon^k+u_\epsilon\cdot\nabla v_\epsilon^k &=&
\Delta v_\epsilon^k-2\chi_{\{1<|d_\epsilon|\le k\}}\big(|\nabla d_\epsilon|^2+\frac{1}{\epsilon^2}(|d_\epsilon|^2-1)|d_\epsilon|^2\big)
\nonumber\\
&\le& \Delta v_\epsilon \ \ {\rm{in}}\ \Omega\times [0,T]
\end{eqnarray}
in the weak sense.
Since $v_\epsilon^k=0$ on $(\Omega\times \{0\})\cup(\partial\Omega\times (0,T])$ and $0\le v_\epsilon^k\le k^2-1$ in $\Omega\times [0,T]$,
we can multiply the (\ref{2.2}) by $v_\epsilon^k$ and integrate it over $\Omega\times [0,s]$ for any $0<s\le T$ to obtain
\begin{eqnarray*}
\int_\Omega |v_\epsilon^k(s)|^2+2\int_0^s\int_\Omega |\nabla v_\epsilon^k|^2
&\le& -\int_0^s\int_\Omega u_\epsilon\cdot\nabla (|v_\epsilon^k|^2)=0,
\end{eqnarray*}
where we have used the fact that $\nabla \cdot u_\epsilon=0$ in the last step.
Hence it follows that $v_\epsilon^k=0$ a.e. in $\Omega\times [0,T]$ and hence $|d_\epsilon|\le 1$ a.e. in $\Omega\times [0,T]$.
\end{proof}

\begin{lemma}\label{MP2}  For $0<T<+\infty$, assume $u_\epsilon\in L^2([0,T], {\bf J})$ and
$g_\epsilon\in H^1(\Omega,\mathbb R^3)$ satisfies
$$|g_\epsilon(x)|\le 1\ {\rm{and}}\ g_\epsilon^3(x)\ge  0, \ {\rm{a.e.}}\ x\in\Omega.$$
If $d_\epsilon\in L^2([0,T], H^1(\Omega,\mathbb R^3))$, with $(1-|d_\epsilon|^2)\in L^2(\Omega\times [0,T])$, solves
(\ref{TGL}), then
$$|d_\epsilon(x,t)|\le 1 \ {\rm{and}}\ d_\epsilon^3(x,t)\ge 0,\  {\rm{a.e.}} \ (x,t)\in\Omega\times [0,T].$$
\end{lemma}
\begin{proof} By Lemma \ref{MP1}, we have that $|d_\epsilon|\le 1$ a.e. in $\Omega\times [0,T]$ and hence
$$0\le \frac{1}{\epsilon^2}(1-|d_\epsilon|^2)\le \frac{1}{\epsilon^2}.$$
Define $\widetilde{d_\epsilon^3}(x,t)=e^{-\frac{t}{\epsilon^2}} d_\epsilon^3(x,t)$. Then we have
$$\partial_t\widetilde{d_\epsilon^3}+u_\epsilon\cdot\nabla \widetilde{d_\epsilon^3}
-\Delta\widetilde{d_\epsilon^3}=\big(\frac{1}{\epsilon^2}(1-|d_\epsilon|^2)-\frac{1}{\epsilon^2}\big)\widetilde{d_\epsilon^3}
\equiv c_\epsilon(x,t)\widetilde {d_\epsilon^3},$$
where $c_\epsilon(x,t):=\big(\frac{1}{\epsilon^2}(1-|d_\epsilon|^2)-\frac{1}{\epsilon^2}\big)(x,t)$ satisfies
$$c_\epsilon\le 0 \ {\rm{a.e.}}\ \Omega\times [0,T].$$
Define $(\widetilde{d_\epsilon^3})^{-}=-\min\big\{\widetilde{d_\epsilon^3}, 0\big\}$ in $\Omega\times [0,T]$.
Then we have
\begin{equation}\label{2.3}
\partial_t(\widetilde{d_\epsilon^3})^{-}+u_\epsilon\cdot\nabla (\widetilde{d_\epsilon^3})^{-}
-\Delta(\widetilde{d_\epsilon^3})^{-}=c_\epsilon(x,t)(\widetilde {d_\epsilon^3})^{-}.
\end{equation}
Since
$$\widetilde{d_\epsilon^3}=e^{-\frac{t}{\epsilon^2}}g_\epsilon^3\ge 0\  {\rm{on}}\ (\Omega\times\{0\})\cup
(\partial\Omega\times [0,T]),$$
it follows that
$$(\widetilde{d_\epsilon^3})^{-}=0 \ {\rm{on}}\ (\Omega\times \{0\})\cup(\partial\Omega \times [0,T]).$$
Multiplying (\ref{2.3}) by $(\widetilde{d_\epsilon^3})^{-}$, integrating the resulting equation over
$\Omega\times [0,s]$ for $0<s\le T$, and using the fact that $\nabla\cdot u_\epsilon=0$ and $c_\epsilon\le 0$,
we obtain
\begin{eqnarray*}
\int_\Omega|(\widetilde{d_\epsilon^3})^{-}|^2(s)
+2\int_0^s\int_\Omega |\nabla (\widetilde{d_\epsilon^3})^{-}|^2
&=& -\int_0^s\int_\Omega u_\epsilon\cdot\nabla |(\widetilde{d_\epsilon^3})^{-}|^2
+2\int_0^s\int_\Omega c_\epsilon(x,t)|(\widetilde {d_\epsilon^3})^{-}|^2\\
&=&2\int_0^s\int_\Omega c_\epsilon(x,t)|(\widetilde {d_\epsilon^3})^{-}|^2\le 0.
\end{eqnarray*}
Hence it follows that $(\widetilde{d_\epsilon^3})^{-}=0$ a.e.  $\Omega\times [0,T]$. This implies that
$$d_\epsilon^3\ge 0 \ \ {\rm{a.e.}}\ \ \Omega\times [0,T].$$
This completes the proof.
\end{proof}

\medskip

\section{Monotonicity formula for approximated Ginzburg-Landau equation}

In this section, we will derive the monotonicity formula for approximated Ginzburg-Landau equations with $L^2$-tension fields
in $\Omega\subset\mathbb R^3$.

\begin{lemma} \label{AGL0}
Let $d_\epsilon\in H^1(\Omega, \mathbb R^3)$ be a solution of the approximated Ginzburg-Landau equation:
\begin{equation}\label{AGL}
\Delta d_\epsilon+\frac{1}{\epsilon^2}(1-|d_\epsilon|^2) d_\epsilon=\tau_\epsilon\ \ {\rm{in}}\ \ \Omega.
\end{equation}
Assume $|d_\epsilon|\le 1$ a.e. $\Omega$ and $\tau_\epsilon\in L^2(\Omega)$. Then, for any $x_0\in\Omega$ and
$0<r\le R<{\rm{dist}}(x_0,\partial\Omega)$, it holds
\begin{equation}\label{AMI1}
\Phi^\epsilon(R)\ge \Phi^\epsilon(r)+\frac12\int_{B_R(x_0)\setminus B_r(x_0)}
 \frac{1}{|x-x_0|}\Big(\big|\frac{\partial d_\epsilon}{\partial |x-x_0|}\big|^2+\frac{(1-|d_\epsilon|^2)^2}{\epsilon^2}\Big),
\end{equation}
where
\begin{equation}\label{RNE}
\Phi^\epsilon(\rho)=\frac1{\rho}\int_{B_\rho(x_0)}\big(e_\epsilon(d_\epsilon)-\langle (x-x_0)\cdot\nabla d_\epsilon, \tau_\epsilon\rangle\big)\
+\frac12\int_{B_\rho(x_0)}|x-x_0||\tau_\epsilon|^2
\end{equation}
for $\rho>0$, and $e_\epsilon(d_\epsilon)=\big(\frac12{|\nabla d_\epsilon|^2}+\frac{3}{4\epsilon^2}
(1-|d_\epsilon|^2)^2\big)$ denotes the (modified) Ginzburg-Landau
energy density of $d_\epsilon$.
\end{lemma}
\begin{proof} Since $|d_\epsilon|\le 1$ a.e. $\Omega$, we have that
$$\Big\|\tau_\epsilon+\frac{1}{\epsilon^2}(|d_\epsilon|^2-1)d_\epsilon\Big\|_{L^2(\Omega)}
\le \big\|\tau_\epsilon\big\|_{L^2(\Omega)}+\frac{|\Omega|}{\epsilon^2}.$$
Hence by the $W^{2,2}$-estimate,  we have that $d_\epsilon\in W^{2,2}_{\rm{loc}}(\Omega)$.

For simplicity, assume $x_0=0\in\Omega$ and write $d$ for $d_\epsilon$.
Multiplying the equation (\ref{AGL}) by $x\cdot\nabla d$ and integrating over $B_\rho\subset\Omega$ yields
$$\rho\int_{\partial B_\rho}\Big(\big|\frac{\partial d}{\partial |x|}\big|^2+\frac{(1-|d|^2)^2}{2\epsilon^2}\Big)+\int_{B_\rho}e_\epsilon(d)-\rho\int_{\partial B_\rho}e_\epsilon(d)
=\int_{B_\rho}\langle \tau_\epsilon, x\cdot\nabla d\rangle.
$$
This implies
\begin{eqnarray*}
&&\frac{d}{d\rho}\Big[\frac1{\rho}\int_{B_\rho}(e_\epsilon(d)-\langle \tau_\epsilon, x\cdot\nabla d\rangle)\Big]\\
&=&\frac1{\rho^2}\Big[\rho\int_{\partial B_\rho}e_\epsilon(d)-\int_{B_\rho}e_\epsilon(d)+\int_{B_\rho}\langle \tau_\epsilon, x\cdot\nabla d\rangle\Big]
-\frac1{\rho}\int_{\partial B_\rho}\langle \tau_\epsilon, x\cdot\nabla d\rangle
\\
&=&\frac1{\rho}\int_{\partial B_\rho}\Big(\big|\frac{\partial d}{\partial |x|}\big|^2
+\frac{(1-|d|^2)^2}{2\epsilon^2}\Big)-\frac1{\rho}\int_{\partial B_\rho}\langle \tau_\epsilon, x\cdot\nabla d\rangle.
\end{eqnarray*}
By H\"older's inequality, we have
$$\Big|\frac1{\rho}\int_{\partial B_\rho}\langle \tau_\epsilon, x\cdot\nabla d\rangle\Big|
\le \frac1{2\rho}\int_{\partial B_\rho}\big|\frac{\partial d}{\partial |x|}\big|^2
+\frac12 \rho\int_{\partial B_\rho}|\tau_\epsilon|^2. $$
Thus we obtain
$$
\frac{d}{d\rho}\Big[\frac1{\rho}\int_{B_\rho}(e_\epsilon(d)-\langle \tau_\epsilon, x\cdot\nabla d\rangle)\Big]
\ge
\frac1{2\rho}\int_{\partial B_\rho}\Big(\big|\frac{\partial d}{\partial |x|}\big|^2
+\frac{(1-|d|^2)^2}{\epsilon^2}\Big)-\frac12\rho\int_{\partial B_\rho}|\tau_\epsilon|^2.
$$
Integrating this inequality over $r\le\rho\le R$ yields (\ref{AMI1}).
\end{proof}

\medskip
\section{$\delta_0$-compactness property of approximated Ginzburg-Landau equation}
In this section, we will prove an $\delta_0$-regularity property for approximated Ginzburg-Landau equations with $L^2$-tension fields  in $\Omega\subset\mathbb R^3$.

First we need to recall some notations. For $1\le p<+\infty$, $0\le q\le 3$, and an open set $U\subset\mathbb R^3$,
we define the Morrey space $M^{p,q}(U)$  by
\begin{equation}\label{morrey_space}
M^{p,q}(U):=\Big\{f\in L^p_{\rm{loc}}(U)\ \Big|
\ \big\|f\big\|_{M^{p,q}(U)}^q\equiv\sup_{B_r\subset U} \frac1{r^{3-q}}\int_{B_r}|f|^p<+\infty\Big\}.
\end{equation}

Now we consider approximated Ginzburg-Landau equations with $L^2$-tension fields.
For $0<\epsilon\le 1$, let $d_\epsilon\in H^1(\Omega, \mathbb R^3)$ be a sequence of solutions to
\begin{equation}\label{AGL2}
\Delta d_\epsilon+\frac{1}{\epsilon^2}(1-|d_\epsilon|^2)d_\epsilon=\tau_\epsilon \ \ {\rm{in}}\ \ \Omega,
\end{equation}
 with uniformly bounded Ginzburg-Landau
energies and $L^2$-norms of $\tau_\epsilon$, i.e.,
\begin{equation}
\sup_{0<\epsilon\le 1}E_\epsilon(d_\epsilon)=\int_\Omega \big(\frac12|\nabla d_\epsilon|^2+\frac{1}{4\epsilon^2}(1-|d_\epsilon|^2)^2\big)\le L_1<+\infty,\label{bounded_GL}
\end{equation}
and
\begin{equation}\label{tension-2-bound}
\sup_{0<\epsilon\le 1}\big\|\tau_\epsilon\big\|_{L^2(\Omega)}\le L_2<+\infty.
\end{equation}
After taking a possible subsequence, we may assume that there exists $d\in H^1(\Omega,\mathbb S^2)$ such that
$$d_\epsilon\rightharpoonup d \ {\rm{in}}\ H^1(\Omega) \ {\rm{and\ strongly\ in}}\  L^2(\Omega),$$
as $\epsilon\rightarrow 0$.

A crucial observation we make is the strong convergence of $d_\epsilon$ to $d$
in $H^1$ under the smallness condition of renormalized
Ginzburg-Landau energies. More precisely, we have
\begin{lemma}\label{small_holder}
For any $L_1, L_2>0$, there exist $\delta_0>0$  and $r_0>0$
such that for $0<\epsilon\le 1$, if $d_\epsilon\in H^1(\Omega,\mathbb R^3)$, with $|d_\epsilon|\le 1$
a.e. $\Omega$,  is a family of solutions of (\ref{AGL2})
satisfying (\ref{bounded_GL}) and  (\ref{tension-2-bound}), and
\begin{equation}\label{small1}
\frac1{r_1}\int_{B_{r_1}(x_0)}e_\epsilon(d_\epsilon)\le \delta_0^2,
\end{equation}
for some $x_0\in\Omega$ and $0<r_1\le\frac12\min\{r_0, {\rm{dist}}(x_0,\partial\Omega)\}$, then
 after passing to subsequences,
$d_\epsilon\rightarrow d$ in $H^1(B_{\frac{r_1}4}(x_0),\mathbb R^3)$ as $\epsilon\rightarrow 0$.
\end{lemma}
\begin{proof} For simplicity, assume $x_0=0\in\Omega$. For any fixed $x_1\in B_{\frac{r_1}2}$ and $0<\epsilon\le \frac{r_1}2$,
define $\widehat{d_\epsilon}(x)=d_\epsilon(x_1+\epsilon x): B_2\to\mathbb R^3$.
Then we have
$$
\Delta \widehat{d_\epsilon}=-(1-|\widehat{d_\epsilon}|^2)\widehat{d_\epsilon}+\widehat{\tau_\epsilon} \ \ {\rm{in}}\ \ B_2,
$$
where $\widehat{\tau_\epsilon}(x)=\epsilon^2\tau_\epsilon(x_1+\epsilon x)$. Since $|\widehat{d_\epsilon}|=|d_\epsilon|\le 1$
in $\Omega$,
it is easy to see that
$$\big\|\Delta\widehat{d_\epsilon}\big\|_{L^2(B_2)}\le \big\|(1-|\widehat{d_\epsilon}|^2)\widehat{d_\epsilon}\big\|_{L^2(B_2)}+\big\|\widehat{\tau_\epsilon}\big\|_{L^2(B_2)}\le C+\epsilon^\frac12\big\|\tau_\epsilon\big\|_{L^2(\Omega)}
\le C+L_2.
$$
Thus $\widehat{d_\epsilon}\in W^{2,2}(B_1)$ and
$$\big\|\widehat{d_\epsilon}\big\|_{W^{2,2}(B_1)}\leq C\Big[\big\|\widehat{d_\epsilon}\big\|_{L^2(B_2)}+\big\|\Delta\widehat{d_\epsilon}\big\|_{L^2(B_2)}\Big]
\le C(1+L_2).
$$
By Sobolev's embedding theorem, we conclude that $\widehat{d_\epsilon}\in C^\frac12(B_1)$ and
$$\Big[\widehat{d_\epsilon}\Big]_{C^\frac12(B_1)}\le C\Big\|\widehat{d_\epsilon}\Big\|_{W^{2,2}(B_1)}\le C(1+L_2).$$
Scaling back to the original scales, this implies that
$$\big|d_\epsilon(x)-d_\epsilon(y)\big|
\le C(1+L_2)\Big(\frac{|x-y|}{\epsilon}\Big)^\frac12, \ \forall\ x, y\in B_\epsilon(x_1).$$
Now we have

\smallskip
\noindent{\bf Claim 4.1}. $|d_\epsilon(x)|\ge \frac12$ for $x\in B_{\frac{r_1}2}$.

Suppose that the claim were false. Then there exists $x_1\in B_{\frac{r_1}2}$ such that
$\displaystyle|d_\epsilon(x_1)|<\frac12.$
Then for any $\theta_0\in (0,1)$ and $x\in B_{\theta_0\epsilon}(x_1)$, it holds
$$|d_\epsilon(x)-d_\epsilon(x_1)|\le  C\Big(\frac{|x-x_1|}{\epsilon}\Big)^\frac12\le C\theta_0^\frac12<\frac14,$$
provided $\theta_0<\frac{1}{16C^2}$. Hence we have
$$|d_\epsilon(x)|\le \frac34, \ \forall\ x\in B_{\theta_0\epsilon}(x_1),$$
so that
\begin{equation}\label{lower_bd}
\frac1{\theta_0\epsilon}\int_{B_{\theta_0\epsilon}(x_1)}\frac{(1-|d_\epsilon|^2)^2}{4\epsilon^2}
\ge \big(\frac{7}{32}\big)^2\frac{\big|B_{\theta_0\epsilon}(x_1)\big|}{\theta_0\epsilon^3}\ge \big(\frac{7}{32}\big)^2\theta_0^2
\big|B_1\big|.
\end{equation}
On the other hand, by the monotonicity inequality (\ref{AMI1}) we have
\begin{eqnarray}\label{upper_bd}
\frac1{\theta_0\epsilon}\int_{B_{\theta_0\epsilon}(x_1)}\frac{(1-|d_\epsilon|^2)^2}{4\epsilon^2}
&\le &\frac{1}{\theta_0\epsilon}\int_{B_{\theta_0\epsilon}(x_1)}e_\epsilon(d_\epsilon)\nonumber\\
&\le & C\frac1{r_1}\int_{B_{r_1}(x_1)}e_\epsilon(d_\epsilon)+C\int_{B_{r_1}(x_1)}|x-x_1||\tau_\epsilon|^2\nonumber\\
&\le& C(\delta_0^2+r_1\big\|\tau_\epsilon\big\|_{L^2(\Omega)}^2)\nonumber\\
&\le& C(\delta_0^2+L_2^2r_1).
\end{eqnarray}
It is clear that (\ref{lower_bd}) contradicts (\ref{upper_bd}), provided
$r_0>0$ and $\delta_0>0$ are chosen to be sufficiently small.
This yields the conclusion of claim 4.1.

Since $|d_\epsilon|\ge \frac12$ in $B_{\frac{r_1}2}$, we can perform the polar decomposition
of $d_\epsilon$ by $d_\epsilon=f_\epsilon\omega_\epsilon$, where
$$f_\epsilon:=|d_\epsilon|:B_{\frac{r_1}2}\to [\frac12, 1]\  \ {\rm{and}} \ \
\omega_\epsilon:=\frac{d_\epsilon}{|d_\epsilon|}: B_{\frac{r_1}2}\to \mathbb S^2.$$
Denote the cross product in $\mathbb R^3$ by $\times$.
It is readily seen that  $$\nabla d_\epsilon =(\nabla f_\epsilon) \omega_\epsilon +f_\epsilon \nabla \omega_\epsilon,
\ \ \ \nabla d_\epsilon\times d_\epsilon=f_\epsilon^2\nabla\omega_\epsilon\times\omega_\epsilon
\ {\rm{in}}\ B_{\frac{r_1}2}.$$
Hence we have that for any subset $U\subset B_{\frac{r_1}2}$, it holds
$$\big\|\nabla f_\epsilon\big\|_{M^{2,2}(U)}+\big\|\nabla \omega_\epsilon\big\|_{M^{2,2}(U)}
\le C\big\|\nabla d_\epsilon\big\|_{M^{2,2}(U)}.
$$
Now we have

\smallskip
\noindent{\bf Claim 4.2}.  For any $2<p<3$, $\nabla\omega_\epsilon\in L^p(B_{\frac{r_1}2})$ and
\begin{equation}\label{lp-bound}
\big\|\nabla \omega_\epsilon\big\|_{L^p(B_{\frac{r_1}2})}
\le C\Big(\big\|\nabla\omega_\epsilon\big\|_{L^{2}(B_{r_1})}+\big\|\tau_\epsilon\big\|_{L^2(B_{r_1})}\Big).
\end{equation}

 Let $d$ and $d^*$ denote exterior derivative and
co-exterior derivative respectively. It follows directly from (\ref{AGL2}) that
\begin{equation}\label{AGL3}
d^*(d d_\epsilon\times d_\epsilon)=\tau_\epsilon\times d_\epsilon\ \ {\rm{in}}\ \ B_{\frac{r_1}2}.
\end{equation}
For any ball $B_r\subset B_{\frac{r_1}4}$, let $\widetilde{f_\epsilon}:\mathbb R^3\to \mathbb R$ and
$\widetilde{\omega}_\epsilon:\mathbb R^3\to \mathbb R^3$ be extensions of $f_\epsilon$ and $\omega_\epsilon
$ in $B_{2r}$  such that
\begin{equation}\label{extension}
\begin{cases}
\widetilde f_\epsilon=f_\epsilon \ {\rm{in}}\ B_{2r},\ \
\frac12\le \widetilde f_\epsilon\le 1 \ {\rm{in}}\ \mathbb R^3,\ \ \big\|\nabla \widetilde f_\epsilon\big\|_{L^2(\mathbb R^3)}
\le C \big\|\nabla f_\epsilon\big\|_{L^2(B_{2r})},\\
\widetilde \omega_\epsilon=\omega_\epsilon \ {\rm{in}}\ B_{2r},\ \ \ |\widetilde\omega_\epsilon|\le 1 \ {\rm{in}}\ \mathbb R^3,
\ \ \big\|\nabla\widetilde{\omega}_\epsilon\big\|_{L^2(\mathbb R^3)}
\le C\big\|\nabla\omega_\epsilon\big\|_{L^2(B_{2r})},\\
\big\|\nabla\widetilde f_\epsilon\big\|_{M^{2,2}(\mathbb R^3)}
\le C\big\|\nabla f_\epsilon\big\|_{M^{2,2}(B_{2r})},\ \ \big\|\nabla\widetilde \omega_\epsilon\big\|_{M^{2,2}(\mathbb R^3)}
\le C\big\|\nabla \omega_\epsilon\big\|_{M^{2,2}(B_{2r})}.
\end{cases}
\end{equation}
Set  $\widetilde d_\epsilon=\widetilde f_\epsilon \widetilde\omega_\epsilon$ in $\mathbb R^3$.
Applying the Hodge decomposition theorem to $d\widetilde d_\epsilon\times \widetilde d_\epsilon=\widetilde{f_\epsilon}^2d\widetilde \omega_\epsilon\times \widetilde \omega_\epsilon$,
we conclude that there exist
$G^\epsilon\in \dot{H}^1(\mathbb R^3, \mathbb M^{3\times 3})$ and $H^\epsilon\in \dot{H}^1(\mathbb R^3,
\wedge^2(\mathbb M^{3\times 3}))$ such that
\begin{equation}\label{hodge}
d\widetilde d_\epsilon\times \widetilde d_\epsilon
=d G^\epsilon+d^*H^\epsilon, \ dH^\epsilon=0 \ \ {\rm{in}}\ \mathbb R^3,
\end{equation}
and
\begin{equation}\label{2-bound}
\big\|\nabla G^\epsilon\big\|_{L^2(\mathbb R^3)}+\big\|\nabla H^\epsilon\big\|_{L^2(\mathbb R^3)}\lesssim
\big\|d\widetilde d_\epsilon\times \widetilde d_\epsilon\big\|_{L^2(\mathbb R^3)}
\lesssim \big\|\nabla\widetilde\omega_\epsilon\big\|_{L^2(\mathbb R^3)}
\lesssim \big\|\nabla \omega_\epsilon\big\|_{L^2(B_{2r})}.
\end{equation}
Taking the exterior derivative of both sides of the equation (\ref{hodge}), we have
\begin{equation}\label{H-equation}
\Delta H^\epsilon=d(d\widetilde d_\epsilon\times \widetilde d_\epsilon)=d\widetilde{d}_\epsilon\times d\widetilde{d}_\epsilon
=\sum_{i,j=1}^3\frac{\partial \widetilde{d_\epsilon}}{\partial x_i}\times
\frac{\partial \widetilde{d_\epsilon}}{\partial x_j} dx_i\wedge dx_j
 \ {\rm{in}}\ \mathbb R^3.
\end{equation}
Multiplying (\ref{H-equation}) by $H^\epsilon$ and applying integration by parts, we have
\begin{eqnarray*}
\int_{\mathbb R^3} |\nabla H^\epsilon|^2=
-\int_{\mathbb R^3}d \widetilde d_\epsilon\times d \widetilde d_\epsilon \cdot H^\epsilon
=-\int_{\mathbb R^3}d\widetilde d_\epsilon\times \widetilde d_\epsilon \cdot d^*H^\epsilon
=-\int_{\mathbb R^3}d\widetilde \omega_\epsilon\times ({\widetilde f_\epsilon}^2\widetilde\omega_\epsilon)
\cdot d^*H^\epsilon.
\end{eqnarray*}
Applying the duality between the Hardy space $\mathcal H^1(\mathbb R^3)$
and the BMO space ${\rm{BMO}}(\mathbb R^3)$ (see \cite{evans} \cite{helein} or \cite{LW4}), we then obtain
\begin{eqnarray*}
\int_{\mathbb R^3} |\nabla H^\epsilon|^2
&\lesssim& \big\|d\widetilde {\omega}_\epsilon\cdot d^*H^\epsilon
\big\|_{\mathcal {H}^1(\mathbb R^3)}
\big\|{\widetilde f_\epsilon}^2\widetilde \omega_\epsilon\big\|_{\rm{BMO}(\mathbb R^3)}\\
&\lesssim& \big\|\nabla \widetilde \omega_\epsilon\big\|_{L^2(\mathbb R^3)}
\big\|\nabla H^\epsilon\big\|_{L^2(\mathbb R^3)}
\big\|\nabla({\widetilde f_\epsilon}^2 \widetilde\omega_\epsilon)\big\|_{M^{2,2}(\mathbb R^3)}\\
&\lesssim& \big\|\nabla \omega_\epsilon\big\|_{L^2(B_{2r})}^2
\Big[\big\|\nabla \widetilde f_\epsilon\big\|_{M^{2,2}(\mathbb R^3)}
+\big\|\nabla \widetilde \omega_\epsilon\big\|_{M^{2,2}(\mathbb R^3)}\Big]\\
&\lesssim& \big\|\nabla \omega_\epsilon\big\|_{L^2(B_{2r})}^2
\Big[\big\|\nabla f_\epsilon\big\|_{M^{2,2}(B_{2r})}
+\big\|\nabla \omega_\epsilon\big\|_{M^{2,2}(B_{2r})}\Big]\\
&\lesssim& \big\|\nabla \omega_\epsilon\big\|_{L^2(B_{2r})}^2
\big\|\nabla  d_\epsilon\big\|_{M^{2,2}(B_{2r})},
\end{eqnarray*}
where we have used the Poincar\'e inequality to estimate
$$\big\|{\widetilde f_\epsilon}^2\widetilde \omega_\epsilon\big\|_{\rm{BMO}(\mathbb R^3)}
\lesssim \big\|\nabla({\widetilde f_\epsilon}^2\widetilde \omega_\epsilon)\big\|_{M^{2,2}(\mathbb R^3)}
\lesssim \Big[\big\|\nabla \widetilde f_\epsilon\big\|_{M^{2,2}(\mathbb R^3)}
+\big\|\nabla \widetilde \omega_\epsilon\big\|_{M^{2,2}(\mathbb R^3)}\Big]
$$
among the first three inequalities.
Utilizing the energy monotonicity inequality (\ref{AMI1}),  we obtain
\begin{eqnarray*}
\big\|\nabla d_\epsilon\big\|_{M^{2,2}(B_{2r})}
\lesssim \Big\{\frac{1}{2r_1}\int_{B_{2r_1}}e_\epsilon(d_\epsilon)+r_1\int_{B_{2r_1}}|\tau_\epsilon|^2 \Big\}^\frac12
\lesssim \delta_0+L_2r_1^\frac12.
\end{eqnarray*}
Thus we have
\begin{equation}\label{H-estimate}
\int_{\mathbb R^3} |\nabla H^\epsilon|^2\le C(\delta_0+L_2r_1^\frac12)\int_{B_{2r}}|\nabla \omega_\epsilon|^2.
\end{equation}
To estimate $G^\epsilon$, first observe that by taking the co-exterior derivative $d^*$ of both sides of
the equation (\ref{hodge}), we have
\begin{equation}\label{hodge1}
\Delta G^\epsilon=\tau_\epsilon\times d_\epsilon \ \ {\rm{in}}\ \ B_{r_1}.
\end{equation}
Decompose $G^\epsilon=G^{(\epsilon,1)}+G^{(\epsilon, 2)}$, where $G^{(\epsilon,1)}\in H^1_0(B_{2r},\mathbb M^{3\times 3})$ solves
$$\begin{cases}
\Delta G^{(\epsilon, 1)}=\tau_\epsilon\times d_\epsilon  &  \ {\rm{in}}\ \ B_{2r}\\
\ \   G^{(\epsilon, 1)}=0 & \ {\rm{on}}\ \ \partial B_{2r},
\end{cases}
$$
and $G^{(\epsilon, 2)}\in H^1(B_{2r},\mathbb M^{3\times 3})$ solves
$$
\begin{cases}
\Delta G^{(\epsilon, 2)}=0 & \ {\rm{in}}\ \ B_{2r}\\
\ \ G^{(\epsilon, 2)}=G^\epsilon & \ {\rm{on}}\ \ \partial B_{2r}.
\end{cases}
$$
By the standard elliptic theory, we have that
$$\int_{B_{2r}}|\nabla G^{(\epsilon, 1)}|^2\lesssim r^2 \int_{B_{2r}}|\tau_\epsilon|^2,$$
and
$$\int_{B_{\theta r}}|\nabla G^{(\epsilon, 2)}|^2\lesssim \theta^3 \int_{B_{2r}}|\nabla G^\epsilon|^2\lesssim
\theta^3\int_{B_{2r}}|\nabla \omega_\epsilon|^2, \ \forall\ 0<\theta<1.$$
Combining these two estimates together yields
\begin{equation}\label{G-estimate}
\int_{B_{\theta r}}|\nabla G^\epsilon|^2\lesssim \theta^3\int_{B_{2r}}|\nabla \omega_\epsilon|^2+
 r^2\int_{B_{2r}}|\tau_\epsilon|^2, \ \forall\ 0<\theta<1.
\end{equation}
Putting (\ref{H-estimate}) and (\ref{G-estimate}) together and using (\ref{hodge}), we obtain that
\begin{equation}\label{decay1}
\frac1{\theta r}\int_{B_{\theta r}}|\nabla \omega_\epsilon|^2\le C\big(\theta^2+\theta^{-1}(\delta_0+L_2 r_1^\frac12)\big)
\frac1{2r}\int_{B_{2r}}|\nabla \omega_\epsilon|^2+C\theta^{-1} r\int_{B_{2r}}|\tau_\epsilon|^2,
\end{equation}
for any $B_{2r}\subset B_{\frac{r_1}2}$ and $0<\theta<1$.

For any $\alpha\in (0,\frac12)$, first choose $\theta=\theta_0\in (0,1)$ such that $C\theta_0^2\le \frac13\theta_0^{2\alpha}$,
 then choose $\delta_0\in (0,1)$ such that $C\delta_0\le  \frac13\theta_0^{2\alpha+1}$, and finally choose
$r_1>0$ such that $CL_2r_1^\frac12\le\frac13\theta_0^{2\alpha+1}$. Thus it follows from (\ref{decay1}) that
\begin{equation}\label{decay2}
\frac{1}{\theta_0 r}\int_{B_{\theta_0 r}}|\nabla \omega_\epsilon|^2\le \theta_0^{2\alpha}
\Big\{\frac{1}{2r}\int_{B_{2r}}|\nabla \omega_\epsilon|^2\Big\}+C_0 \theta_0 r\int_{B_{2r}}|\tau_\epsilon|^2,
\ \forall \ B_{2r}\subset B_{\frac{r_1}2}.
\end{equation}
Iterating (\ref{decay2}) finitely many times yields
\begin{equation}\label{decay3}
\frac{1}{r}\int_{B_r}|\nabla \omega_\epsilon|^2
\le \big(\frac{r}{r_1}\big)^{2\alpha}\Big\{\frac{1}{r_1}\int_{B_{r_1}}|\nabla\omega_\epsilon|^2\Big\}
+C_0 r \int_{B_{r_1}}|\tau_\epsilon|^2, \ \forall \ B_{2r}\subset B_{\frac{r_1}2}.
\end{equation}
Taking supremum over all balls $B_{2r}\subset B_{\frac{r_1}2}$, we obtain that for any $\alpha\in (0, \frac12]$, it holds
\begin{equation}\label{decay4}
\big\|\nabla\omega_\epsilon\big\|_{M^{2,2-2\alpha}(B_{\frac{r_1}2})}^2
\le C(r_1)\Big[\int_{B_{r_1}}|\nabla \omega_\epsilon|^2+\int_{B_{r_1}}|\tau_\epsilon|^2\Big].
\end{equation}
It follows from (\ref{decay4}) that $d(d d_\epsilon\times  d_\epsilon)
=d( f_\epsilon^2 d\omega_\epsilon\times\omega_\epsilon)\in M^{1, 2-\alpha}(B_{\frac{r_1}2})$ and
\begin{eqnarray}\label{morrey_est1}
\big\|d(d d_\epsilon\times  d_\epsilon)\big\|_{M^{1,2-\alpha}(B_{\frac{r_1}2})}
&\lesssim& \big\|\nabla \omega_\epsilon\big\|_{M^{2,2-2\alpha}( B_{\frac{r_1}2})}^2
+\big\|\nabla f_\epsilon\big\|_{M^{2,2}( B_{\frac{r_1}2})}
 \big\|\nabla \omega_\epsilon\big\|_{M^{2,2-2\alpha}(B_{\frac{r_1}2})}\nonumber\\
&\le& C(r_1)\Big[\int_{B_{r_1}}|\nabla \omega_\epsilon|^2+\int_{B_{r_1}}|\tau_\epsilon|^2\Big].
\end{eqnarray}
Based on (\ref{morrey_est1}), we can repeat both the extension and the Hodge decomposition as in (\ref{extension}),
(\ref{hodge}), (\ref{2-bound}), and (\ref{H-equation}). Since
$$
H^\epsilon(x)=\int_{\mathbb R^3} \frac{1}{|x-y|}d(d\widetilde d_\epsilon\times \widetilde d_\epsilon)(y)\,dy, \forall x\in\mathbb R^3,$$
we have
$$|\nabla H^\epsilon(x)|\lesssim \int_{\mathbb R^3} \frac{|d(d\widetilde d_\epsilon\times \widetilde d_\epsilon)|(y)}
{|x-y|^2}\,dy=I_1(|d(d\widetilde d_\epsilon\times \widetilde d_\epsilon)|)(x), \ \forall \ x\in\mathbb R^3,$$
where
$$\displaystyle I_1(f)(x):=\int_{\mathbb R^3} \frac{|f(y)|}{|x-y|^2}\,dy,  \ f\in L^1_{\rm{loc}}(\mathbb R^3),$$
is the Riesz potential of $f$ of order $1$.

Since we can construct the extension $\widetilde d_\epsilon$ of $d_\epsilon$ such that
$$\big\|d(d \widetilde d_\epsilon\times \widetilde d_\epsilon)\big\|_{M^{1,2-\alpha}(\mathbb R^3)}
\le C\big\|d(d d_\epsilon\times  d_\epsilon)\big\|_{M^{1,2-\alpha}(B_{\frac{r_1}2})},$$
we can apply Morrey space estimates of Riesz potentials (see \cite{adams}) to conclude that
$\nabla H^\epsilon\in M^{\frac{2-\alpha}{1-\alpha}, 2-\alpha}_*(\mathbb R^3)$ and
\begin{eqnarray}\label{morrey_est2}
\Big\|\nabla H^\epsilon\Big\|_{M^{\frac{2-\alpha}{1-\alpha}, 2-\alpha}_*(\mathbb R^3)}
&\lesssim& \Big\|d(d \widetilde d_\epsilon\times \widetilde d_\epsilon)\Big\|_{M^{1,2-\alpha}(\mathbb R^3)}\nonumber\\
&\lesssim& C(r_1)\Big[\int_{B_{r_1}}|\nabla \omega_\epsilon|^2+\int_{B_{r_1}}|\tau_\epsilon|^2\Big].
\end{eqnarray}
Since $\displaystyle \lim_{\alpha\uparrow \frac12}\frac{2-\alpha}{1-\alpha}=3$, it follows from
(\ref{morrey_est2}) that $\nabla H^\epsilon\in L^p(B_{\frac{r_1}2})$ for any $2<p<3$, and
\begin{equation}
\label{H-estimate1}
\Big\|\nabla H^\epsilon\Big\|_{L^p(B_{\frac{r_1}2})}\le C\Big\|\nabla H^\epsilon\Big\|_{M^{\frac{2-\alpha}{1-\alpha}, 2-\alpha}_*(\mathbb R^3)}\le  C(r_1)\Big[\int_{B_{r_1}}|\nabla \omega_\epsilon|^2+\int_{B_{r_1}}|\tau_\epsilon|^2\Big].
\end{equation}
On the other hand, applying $W^{2,2}$-estimate of the equation (\ref{hodge1}) we conclude
that $G^\epsilon\in W^{2,2}(B_{\frac{r_1}2})$, and
\begin{equation}\label{G-estimate1}
\big\|\nabla G^\epsilon\big\|_{L^6(B_{\frac{r_1}2})}\lesssim\big\|\nabla G^\epsilon\big\|_{H^1(B_{\frac{r_1}2})}
\lesssim \big\|\nabla G^\epsilon\big\|_{L^2(B_{r_1})}
+\big\|\tau_\epsilon\big\|_{L^2(B_{r_1})}\le C(r_1)
\Big\{\int_{B_{r_1}}(|\nabla d_\epsilon|^2+|\tau_\epsilon|^2)\Big\}^\frac12.
\end{equation}
Combining (\ref{H-estimate1}) and (\ref{G-estimate1}) yields that $\nabla \omega_\epsilon\in L^p(B_{\frac{r_1}2})$
for any $2<p<3$, and the estimate (\ref{lp-bound}) holds.

\smallskip
\noindent{\bf Claim 4.3}. There is a map $\omega\in H^1(B_{\frac{r_1}2},\mathbb S^2)$ such that after taking possible subsequences, $\omega_\epsilon\rightarrow \omega$ in $H^1(B_{\frac{2r_1}5})$ as $\epsilon\rightarrow 0$.

After passing to possible subsequences, we may assume that $f_\epsilon\rightharpoonup
1$ in $H^1(B_{\frac{r_1}2})$ and   there exists $\omega\in H^1(B_{\frac{r_1}2},\mathbb S^2)$
such that $\omega_\epsilon\rightharpoonup\omega$ in $H^1(B_{\frac{r_1}2})$. We may also assume that
$\tau_\epsilon\rightharpoonup\tau$ in $L^2(\Omega)$ for some $\tau\in L^2(\Omega,\mathbb R^3)$.
We also recall that (\ref{decay4}) and Morrey's decay lemma (see \cite{morrey}) imply
$\omega_\epsilon\in C^\frac12(B_{\frac{r_1}2})$ and
\begin{equation}\label{holder_est}
\big[\omega_\epsilon\big]_{C^\frac12(B_{\frac{r_1}2})}\le
C(r_1)\Big(\|\nabla\omega_\epsilon\|_{L^2(B_{r_1})}+\|\tau_\epsilon\|_{L^2(B_{r_1})}\Big)\le C.
\end{equation}
Hence we can assume that
\begin{equation}\label{uniform_conv}
\lim_{\epsilon\rightarrow 0}\big\|\omega_\epsilon-\omega\big\|_{L^\infty(B_{\frac{r_1}2})}=0.
\end{equation}
It follows from Claim 4.2 and (\ref{lp-bound}) that for any $2<p<3$,
\begin{equation}\label{lp_bound1}
\|\nabla\omega_\epsilon\|_{L^p(B_{\frac{r_1}2})}\le C.
\end{equation}
Direct calculations imply that $\omega_\epsilon$ satisfies the equation
\begin{equation}\label{omega-eqn1}
\Delta\omega_\epsilon=\widetilde{\tau_\epsilon}:=-|\nabla\omega_\epsilon|^2\omega_\epsilon-2f_\epsilon^{-1}\nabla f_\epsilon\cdot\nabla\omega_\epsilon
+(\tau_\epsilon-\langle\tau_\epsilon, \omega_\epsilon\rangle\omega_\epsilon),
\ {\rm{in}}\ B_{\frac{r_1}2}.
\end{equation}
Since $f_\epsilon\ge \frac12$ in $B_{\frac{r_1}2}$, we have
$$\big\|\widetilde{\tau_\epsilon}\big\|_{L^{\frac{2p}{p+2}}(B_{\frac{r_1}2})}
\le C(r_1)\Big[\|\tau_\epsilon\|_{L^2(B_{\frac{r_1}2})}+\|\nabla f_\epsilon\|_{L^2(B_{\frac{r_1}2})}\|\nabla\omega_\epsilon\|_{L^p(B_{\frac{r_1}2})}\Big]\le C(r_1, p, L_1, L_2).$$
It follows from the $W^{2,\frac{2p}{p+2}}$-theory that $\omega_\epsilon\in W^{2,\frac{2p}{p+2}}(B_{\frac{2r_1}5})$
and
\begin{equation}\label{2p_bound}
\big\|\omega_\epsilon\big\|_{W^{2,\frac{2p}{p+2}}(B_{\frac{2r_1}5})}
\le C\Big(\|\nabla\omega_\epsilon\|_{L^2(B_{\frac{r_1}2})}+\big\|\widetilde{\tau_\epsilon}\big\|_{L^{\frac{2p}{p+2}}(B_{\frac{r_1}2})}
\Big)
\le C(r_1,p, L_1,L_2).
\end{equation}
It follows from (\ref{lp_bound1}), (\ref{2p_bound}), and the compact embedding
of $W^{2,\frac{2p}{p+2}}\subset W^{1,\frac{2p}{p+2}}$ that
$\nabla\omega_\epsilon\rightarrow \nabla\omega$
in $L^2(B_{\frac{2r_1}5})$.

\smallskip
\noindent{\bf Claim 4.4}. After passing to possible subsequences, $f_\epsilon\rightarrow 1$ in $H^1(B_{\frac{r_1}4})$.

To see this, we need to estimate $\displaystyle\int_{B_{\frac{r_1}4}}|\nabla f_\epsilon|^2$. First,
observe that $f_\epsilon$ satisfies
\begin{equation}\label{g-equation}
\Delta (1-f_\epsilon)-\frac{1}{\epsilon^2}(1-f_\epsilon^2)f_\epsilon
=-|\nabla\omega_\epsilon|^2 f_\epsilon-\tau_\epsilon\cdot\omega_\epsilon, \ \ {\rm{in}}\ \ B_{\frac{r_1}2}.
\end{equation}
By Fubini's theorem, there exists $r_2\in (\frac{r_1}4, \frac{2r_1}5)$ such that
\begin{equation}\label{fubini1}
\int_{\partial B_{r_2}}e_\epsilon(d_\epsilon)\,dH^2
\le \frac8{r_1}\int_{B_{\frac{r_1}2}}e_\epsilon(d_\epsilon).
\end{equation}
Since $|d_\epsilon|\le 1$ in $B_{r_1}$, it is readily seen that for any $2<q<+\infty$,
\begin{equation}\label{lq-bound}
\big\|1-|d_\epsilon|\big\|_{L^q(B_{r_1})}\le \big\|1-|d_\epsilon|\big\|_{L^2(B_{r_1})}^{\frac{2}{q}}
\big\|1-|d_\epsilon|\big\|_{L^\infty(B_{r_1})}^{\frac{q-2}q}
\le C \epsilon^{\frac{2}q}\Big(\int_{B_{r_1}}e_\epsilon(d_\epsilon)\Big)^{\frac{1}{q}}\le C\epsilon^{\frac{2}q}.
\end{equation}
Multiplying (\ref{g-equation}) by $(1-f_\epsilon)$ and  integrating the resulting equation over $B_{r_2}$ and
using $\frac12\le |f_\epsilon|\le 1$ and H\"older's inequality, we obtain
\begin{eqnarray}
&&\int_{B_{r_2}}|\nabla f_\epsilon|^2
+\int_{B_{r_2}}\frac{1}{\epsilon^2}(1-f_\epsilon)^2f_\epsilon(1+f_\epsilon)\nonumber\\
&&=\int_{\partial B_{r_2}}(1-f_\epsilon)\frac{\partial f_\epsilon}{\partial r}\,dH^2
+\int_{B_{r_2}}|\nabla\omega_\epsilon|^2f_\epsilon(1-f_\epsilon)
+\int_{B_{r_2}}\tau_\epsilon\cdot\omega_\epsilon(1-f_\epsilon)\nonumber\\
&&\le \int_{\partial B_{r_2}}(1-|d_\epsilon|)\big|\frac{\partial d_\epsilon}{\partial r}\big|\,dH^2
+\int_{B_{r_2}}|\nabla\omega_\epsilon|^2(1-|d_\epsilon|)+\int_{B_{r_2}}|\tau_\epsilon|(1-|d_\epsilon|)\nonumber\\
&&\lesssim \epsilon^2\frac1{r_1}\int_{B_{\frac{r_1}2}}e_\epsilon(d_\epsilon)
+\big\|\nabla\omega_\epsilon\big\|_{L^p(B_{\frac{r_1}2})}^2\big\|1-|d_\epsilon|\big\|_{L^q(B_{\frac{r_1}2})}^{\frac{1}{q}}
+\epsilon\big\|\tau_\epsilon\big\|_{L^2(B_{\frac{r_1}2})}\Big(\int_{B_{\frac{r_1}2}}e_\epsilon(d_\epsilon)\Big)^\frac12,
\label{decay6}
\end{eqnarray}
where $p\in (2,3)$ and $q=\frac{p}{p-2}\in (3,+\infty)$.
Hence, by (\ref{lp-bound}) and (\ref{lq-bound}) we have
\begin{equation}\label{l2-strong-conv}
\int_{B_{\frac{r_1}4}}\big(|\nabla f_\epsilon|^2+\frac{(1-|d_\epsilon|^2)^2}{4\epsilon^2}\big)
\le C\epsilon^{\frac{1}{q}} \rightarrow 0, \ {\rm{as}}\ \epsilon\rightarrow 0.
\end{equation}
Combining claim 4.3 with claim 4.4, we see that $d_\epsilon=f_\epsilon\omega_\epsilon\rightarrow d=\omega$ in $H^1(B_{\frac{r_1}4})$.
The proof is now complete.
\end{proof}

\medskip
\section{$\delta_0$-regularity for suitable approximated harmonic map to $\mathbb S^2$}

In this section, we will introduce the notion of suitable approximated harmonic maps to $\mathbb S^2$
with $L^2$-tension fields. Then we will establish the sequential compactness property for such approximated harmonic maps
under the energy smallness condition.

Recall that a map $d\in H^1(\Omega,\mathbb S^2)$ is called an approximated harmonic map with $L^2$-tension field,
if there exists $\tau\in L^2(\Omega, \mathbb R^3)$ such that
\begin{equation}\label{AHM1}
\Delta d+|\nabla d|^2 d=\tau \ \ {\rm{in}}\ \ \Omega,
\end{equation}
holds in the sense of distributions.

\begin{defn} \label{suitable} An approximated harmonic map $d\in H^1(\Omega, \mathbb S^2)$, with tension field $\tau\in L^2(\Omega,\mathbb R^3)$, is called a suitable approximated harmonic map, if
\begin{equation}\label{suitable1}
\frac{d}{dt}\Big|_{t=0}\int_\Omega \big(\frac12 |\nabla(d\circ F_t)|^2+\langle \tau, d\circ F_t\rangle\big)=0
\end{equation}
holds for any $F_t(x)=x+tY(x)$, where $Y=(Y^1,Y^2,Y^3)\in C_0^1(\Omega, \mathbb R^3)$.
\end{defn}

Direct calculations imply that (\ref{suitable1}) is equivalent to
\begin{equation}\label{suitable2}
\int_\Omega \Big(\big\langle\frac{\partial d}{\partial x_i}, \frac{\partial d}{\partial x_j}\big\rangle \frac{\partial Y^i}{\partial x_j}
-\frac{|\nabla d|^2}2 {\rm{div}}Y+\big\langle\tau, Y\cdot\nabla d\big\rangle\Big)=0,
\ \forall \ Y\in C_0^\infty(\Omega,\mathbb R^3).
\end{equation}
\begin{rem} {\rm An approximated harmonic map $d\in H^1(\Omega,\mathbb S^2)$, with $L^2$-tension field $\tau$,
is a suitable approximated harmonic map, if $d\in W^{2,2}(\Omega, \mathbb S^2)$. In fact, (\ref{suitable2})
can be obtained by the Pohozaev argument, namely multiplying (\ref{AHM1}) by $Y\cdot\nabla d$ and integrating
the resulting equation over $\Omega$.}
\end{rem}

For suitable approximated harmonic maps, we have the following energy monotonicity inequality.
\begin{lemma}\label{HMI1} Assume $d\in H^1(\Omega,\mathbb S^2)$ is a suitable approximated
harmonic map with tension field $\tau\in L^2(\Omega,\mathbb R^3)$. Then
\begin{equation}\label{HMI2}
\Psi_R(d; x_0)\ge \Psi_r(d; x_0)+\frac12\int_{B_R(x_0)\setminus
B_r(x_0)}|x-x_0|^{-1}\big|\frac{\partial d}{\partial |x-x_0|}\big|^2,
\end{equation}
for $x_0\in \Omega$ and $0<r\le R<{\rm{d}}(x_0,\partial\Omega)$, where
$$\Psi_r(d; x_0)
:=\frac1{r}\int_{B_r(x_0)}\big(\frac12|\nabla d|^2
-\langle (x-x_0)\cdot \nabla d, \tau\rangle\big)
+\frac12\int_{B_r(x_0)} |x-x_0||\tau|^2.$$
\end{lemma}
\begin{proof} Assume $x_0=0$. For $0<r<{\rm{d}}(0,\partial\Omega)$ and $0<\epsilon<r$,
let $\eta_\epsilon(x)=\eta(|x|)\in C_0^\infty(B_r)$ be such that $0\le \eta_\epsilon\le 1$,
$\eta_\epsilon\equiv 1 $ in $B_{(1-\epsilon)r}$, and $\eta\equiv 0$ outside $B_r$. Substituting
$Y(x)=\eta_\epsilon(x) x$ into (\ref{suitable2}) and then sending $\epsilon$ to zero, we obtain
$$-\int_{B_r} \frac12|\nabla d|^2+\int_{B_r}\langle\tau, x\cdot\nabla d\rangle
-r\int_{\partial B_r}\big|\frac{\partial d}{\partial r}\big|^2+r\int_{\partial B_r}\frac12|\nabla d|^2=0.$$
This implies
\begin{eqnarray*}
&&\frac{d}{d\rho}\Big(\frac1{r}\int_{B_r}(\frac12|\nabla d|^2-\langle \tau, x\cdot\nabla d\rangle)\Big)\\
&=&\frac1{r^2}\Big[r\int_{\partial B_r}\frac12|\nabla d|^2-\int_{B_r}\frac12|\nabla d|^2
+\int_{B_r}\langle \tau, x\cdot\nabla d\rangle\Big]
-\frac1{r}\int_{\partial B_r}\langle \tau, x\cdot\nabla d\rangle
\\
&=&\frac1{r}\int_{\partial B_r}\big|\frac{\partial d}{\partial |x|}\big|^2
-\frac1{r}\int_{\partial B_r}\langle \tau, x\cdot\nabla d\rangle\ge
\frac1{2r}\int_{\partial B_r}\big|\frac{\partial d}{\partial |x|}\big|^2-\frac12 r\int_{\partial B_r}|\tau|^2.
\end{eqnarray*}
Integrating this inequality over $[r,R]$ implies (\ref{HMI2}).
\end{proof}

With the monotonicity inequality (\ref{HMI2}), we have the following small energy regularity result.
\begin{lemma}\label{small_holder1} For any $L_1, L_2>0$, there exist $\delta_0>0$, $r_0>0$  such that
if $d\in H^1(\Omega, \mathbb S^2)$ is a suitable approximated harmonic map with tension field $\tau$, that satisfies
\begin{equation}
E(d):=\frac12\int_\Omega |\nabla d|^2\le L_1, \ \big\|\tau\big\|_{L^2(\Omega)}\le L_2,
\end{equation}
and
\begin{equation}\label{small1}
\frac1{r_1}\int_{B_{r_1}(x_0)}|\nabla d|^2\le\delta_0^2,
\end{equation}
for some $x_0\in \Omega$ and $0<r_1\le \frac12\min\big\{r_0, {\rm{d}}(x_0,\partial\Omega)\big\}$, then
$d\in C^\frac12\cap W^{2,2}(B_{\frac{r_1}2}, \mathbb S^2)$, and
\begin{equation}\label{small_holder2}
\big[d\big]_{C^\frac12(B_{\frac{r_1}2})}+\big\|d\big\|_{W^{2,2}(B_{\frac{r_1}8})}
\le C(r_1, \delta_0, L_1, L_2).
\end{equation}
\end{lemma}
\begin{proof} Since the proof is similar to that of Lemma \ref{small_holder}, we only sketch it here.
First, multiplying both sides of (\ref{AHM1}) by $\times d$ yields that $d$ satisfies
\begin{equation}\label{AHM2}
{\rm{div}}(\nabla d\times d)=\tau\times d \ \ \ {\rm{in}}\ \ \ \Omega.
\end{equation}
Then by repeating the argument of the claim 4.2 lines by lines, we can obtain that
for any $\alpha\in (0,\frac12)$, there exists $\theta_0\in (0,1)$ such that
\begin{equation}\label{hm_decay1}
\frac{1}{\theta_0 r}\int_{B_{\theta_0 r}}|\nabla d|^2\le \theta_0^{2\alpha}
\Big\{\frac{1}{2r}\int_{B_{2r}}|\nabla d|^2\Big\}+C_0 \theta_0 r\int_{B_{2r}}|\tau|^2,
\ \forall \ B_{2r}\subset B_{\frac{r_1}2}.
\end{equation}
Iterating (\ref{hm_decay1}) finitely many times, we obtain
\begin{equation}\label{hm_decay2}
\frac{1}{r}\int_{B_r}|\nabla d|^2
\le \big(\frac{r}{r_1}\big)^{2\alpha}\Big\{\frac{1}{r_1}\int_{B_{r_1}}|\nabla d|^2\Big\}
+C_0 r \int_{B_{r_1}}|\tau|^2, \ \forall \ B_{2r}\subset B_{\frac{r_1}2}.
\end{equation}
Taking supremum over all balls $B_{2r}\subset B_{\frac{r_1}2}$, we obtain that for any $\alpha\in (0, \frac12]$, it holds
\begin{equation}\label{hm_decay3}
\big\|\nabla d\big\|_{M^{2,2-2\alpha}(B_{\frac{r_1}2})}^2
\le C(r_1)\Big[\int_{B_{r_1}}|\nabla d|^2+\int_{B_{r_1}}|\tau|^2\Big].
\end{equation}
This, combined with Morrey's lemma (see \cite{morrey}), implies that
$d\in C^\frac12(B_{\frac{r_1}2})$ and
\begin{equation}\label{holder_est}
\big[d\big]_{C^\frac12(B_{\frac{r_1}2})}\le
C(r_1)\Big(\|\nabla d\|_{L^2(B_{r_1})}+\|\tau\|_{L^2(B_{r_1})}\Big)\le C(r_1, \delta_0, L_1, L_2).
\end{equation}
To see the interior $W^{2,2}$-regularity of $d$, we proceed as follows. Let $\eta\in C_0^\infty(B_{\frac{r_1}2})$
be a cut-off function of $B_{\frac{r_1}4}$, i.e., $0\le\eta\le 1$, $\eta\equiv 1$ in $B_{\frac{r_1}4}$, and $|\nabla\eta|
\le \frac{8}{r_1}$. Define $d_1, d_2:\mathbb R^3\to\mathbb R^3$ by
$$d_1(x)=\int_{\mathbb R^3}\frac{(\eta^2|\nabla d|^2 d)(y)}{|x-y|}\,dy,\
d_2(x)=\int_{\mathbb R^3}\frac{(\eta^2\tau)(y)}{|x-y|}\,dy,$$
and $d_3: B_{\frac{r_1}2}\to\mathbb R^3$ by
$$d_3=d-d_1-d_2.$$
Then by Morrey space estimates of Riesz potentials as in claim 4.2 of Lemma \ref{small_holder},
we have that $\nabla d_1\in M_*^{\frac{2-2\alpha}{1-2\alpha}, 2-2\alpha}(\mathbb R^3)$,
$\nabla d_2\in L^6(\mathbb R^3)$, and
\begin{equation}\label{morrey_est6}
\Big\|\nabla d_1\Big\|_{M_*^{\frac{2-2\alpha}{1-2\alpha}, 2-2\alpha}(\mathbb R^3)}
\le C\Big\|\nabla d\Big\|_{M^{2,2-2\alpha}(B_{\frac{r_1}2})}^2\le C(r_1,\delta_0, L_1,L_2),
\end{equation}
and
\begin{equation}\label{l6-estimate-d2}
\Big\|\nabla d_2\Big\|_{L^6(\mathbb R^3)}\le C\big\|\tau\big\|_{L^2(B_{\frac{r_1}2})}.
\end{equation}
Since $\displaystyle\lim_{\alpha\uparrow \frac12}\frac{2-2\alpha}{1-2\alpha}=+\infty$, (\ref{morrey_est6})
implies that $\nabla d_1\in L^q(B_{\frac{r_1}2})$ for any $q\in (1,+\infty)$, and
\begin{equation}\label{lq-estimate-d1}
\Big\|\nabla d_1\Big\|_{L^q(B_{\frac{r_1}2})}\le C(q, r_1, \delta_0, L_1, L_2).
\end{equation}
Since
$$\Delta d_3=0 \ \ \ {\rm{in}}\ \ \ B_{\frac{r_1}4},$$
it follows from the standard theory that $\nabla d_3\in L^4(B_{\frac{r_1}5})$ and
\begin{equation}\label{l4-estimate-d3}
\Big\|\nabla d_3\Big\|_{L^4(B_{\frac{r_1}5})}\le C
\Big(\big\|\nabla d\big\|_{L^2(B_{\frac{r_1}4})}+\big\|\nabla d_1\big\|_{L^2(B_{\frac{r_1}4})}
+\big\|\nabla d_2\big\|_{L^2(B_{\frac{r_1}4})}\Big)\le C(r_1, \delta_0, L_1, L_2).
\end{equation}
Putting (\ref{lq-estimate-d1}), (\ref{l6-estimate-d2}), and (\ref{l4-estimate-d3}) together yields
that $\nabla d\in L^4(B_{\frac{r_1}5})$ and
$$\Big\|\nabla d\Big\|_{L^4(B_{\frac{r_1}5})}\le C(r_1, \delta_0, L_1, L_2).$$
Now we can apply the standard $L^2$-estimate to conclude that $d\in W^{2,2}(B_{\frac{r_1}8})$
with the desired estimate.
\end{proof}

\medskip
\section{$H^1$ pre-compactness for certain approximated Ginzburg-Landau equation}
In this section, we will consider the set of solutions to approximated Ginzburg-Landau
equation with ranges in $\big\{y=(y^1,y^2,y^3)\in \mathbb R^3:
|y|\le 1, y^3\ge {-1+a}\big\}$, with uniformly bounded energies and uniformly bounded $L^2$-tension fields.
We will show that it is precompact in $H^1_{\rm{loc}}(\Omega)$ and uniformly
bounded in $H^2_{\rm{loc}}(\Omega)$.

For any $0<a\le 2$, $L_1$, and $L_2>0$, define the set $\mathbf{X}({L_1, L_2, a;\Omega})$ consisting of maps
$d_\epsilon\in H^1(\Omega,\mathbb R^3)$, $0<\epsilon\le 1$, that are solutions of
\begin{equation}\label{AGL5}
\Delta d_\epsilon+\frac{1}{\epsilon^2}(1-|d_\epsilon|^2) d_\epsilon=\tau_\epsilon\ \ {\rm{in}}\ \ \Omega
\end{equation}
such that the following properties hold:
\begin{itemize}
\item[(i)] $|d_\epsilon|\le 1$ and $d_\epsilon^3\ge -1+a$ for a.e. $x\in\Omega$.
\item[(ii)] $\displaystyle E_\epsilon(d_\epsilon)=\int_\Omega e_\epsilon(d_\epsilon)\,dx\le L_1$.
\item[(iii)] $\displaystyle\big\|\tau_\epsilon\big\|_{L^2(\Omega)}\le L_2$.
\end{itemize}
We have
\begin{thm}\label{precomp1} For any $a\in (0,2],\ L_1>0,$ and $L_2>0$, the set $\mathbf{X}(L_1, L_2,a; \Omega)$ is precompact in
$H^1_{\rm{loc}}(\Omega,\mathbb R^3)$. In particular, if for $\epsilon\rightarrow 0$,
$\{d_\epsilon\}\subset H^1(\Omega,\mathbb R^3)$ is a sequence of maps in
${\bf X}(L_1,L_2, a; \Omega)$,  then there exists
a map $d_0\in H^1(\Omega,\mathbb S^2)$
such that after passing to possible subsequences, $d_\epsilon\rightarrow d_0$ in $H^1_{\rm{loc}}(\Omega, \mathbb R^3)$.
\end{thm}
\begin{proof} For $0<\epsilon_i\le 1$, let $\{d_{\epsilon_i}\}\subset {\bf X}(L_1,L_2, a; \Omega)$ be a sequence of
maps. Assume that there are $\epsilon_0\in [0,1]$ and $d_0\in H^1(\Omega,\mathbb R^3)$
such that $\epsilon_i\rightarrow \epsilon_0$ and $d_{\epsilon_i}\rightharpoonup d_0$ in $H^1(\Omega,\mathbb R^3)$
and $\tau_{\epsilon_i}\rightharpoonup \tau_0$ in $L^2(\Omega, \mathbb R^3)$
as $i\rightarrow +\infty$.
We divide the proof into two cases.\\
{\it Case 1}: $\epsilon_0>0$. Since
$$\Big\|\frac{1}{\epsilon_i^2}(1-|d_{\epsilon_i}|^2)d_{\epsilon_i}\Big\|_{L^\infty(\Omega)}\le \frac{2}{\epsilon_0^2},$$
we have
$$\Big\|\Delta d_{\epsilon_i}\Big\|_{L^2(\Omega)}\le
\Big\|\frac{1}{\epsilon_i^2}(1-|d_{\epsilon_i}|^2)d_{\epsilon_i}\Big\|_{L^2(\Omega)}+\big\|\tau_{\epsilon_i}\big\|_{L^2(\Omega)}
\le C\epsilon_0^{-2}|\Omega|+L_2.$$
By $W^{2,2}$-estimate we conclude that $\{d_{\epsilon_i}\}$ is a bounded sequence in $W^{2,2}_{\rm{loc}}(\Omega)$.
Hence we have that $d_{\epsilon_i}\rightarrow d_0$ in $H^1_{\rm{loc}}(\Omega, \mathbb R^3)$
as $i\rightarrow +\infty$.

\smallskip
\noindent
{\it Case 2}: $\epsilon_0=0$. Then it is easy to see that $d_0\in H^1(\Omega,\mathbb S^2)$.
We may assume that there exists nonnegative Radon measures $\nu$ and $\mu$ in $\Omega$ such that
$$e_{\epsilon_i}(d_{\epsilon_i})\,dx\rightharpoonup \mu:=\frac12|\nabla d_0|^2\,dx+\nu$$
as convergence of Radon measures in $\Omega$ for $i\rightarrow +\infty$.

Let $\delta_0>0$ be given by
Lemma \ref{small_holder}. Define the concentration set $\Sigma\subset\Omega$ by
\begin{equation}\label{concentration}
\Sigma:=\bigcap_{0<r<{\rm{d}}(x_0,\partial\Omega)}\Big\{x_0\in \Omega:
\ \liminf_{i\rightarrow +\infty} \Phi_r(d_{\epsilon_i}; x_0)\ge \delta_0^2\Big\},
\end{equation}
where
$$\Phi_r(d_{\epsilon_i}; x_0)
:=\frac1{r}\int_{B_r(x_0)}\big(e_{\epsilon_i}(d_{\epsilon_i})-\langle (x-x_0)\cdot \nabla d_{\epsilon_i},\tau_{\epsilon_i}\rangle\big)
+\frac12\int_{B_r(x_0)} |x-x_0||\tau_{\epsilon_i}|^2.$$
By Lemma \ref{AGL0}, we have that for any  $x_0\in\Omega, \ 0<r\le R<{\rm{d}}(x_0,\partial\Omega)$,
it holds
\begin{equation}\label{AMI3}
\Phi_R(d_{\epsilon_i}; x_0)\ge \Phi_r(d_{\epsilon_i}; x_0)+\frac12\int_{B_R(x_0)\setminus
B_r(x_0)}|x-x_0|^{-1}\big|\frac{\partial d_{\epsilon_i}}{\partial |x-x_0|}\big|^2.
\end{equation}

Motivated by the blow-up analysis for stationary harmonic maps by Lin \cite{Lin} and harmonic
map heat flows by Lin-Wang \cite{LW1,LW2, LW3} (see also \cite{LW4}),  we will perform the blow-up analysis of the approximated
harmonic maps in ${\bf X}(L_1, L_2, a; \Omega)$.

Since
$$\Big|\frac1{r}\int_{B_r(x_0)}\langle (x-x_0)\cdot\nabla d_{\epsilon_i},\tau_{\epsilon_i}\rangle \Big|
\le \frac{r^{\frac12}}{r}\int_{B_r(x_0)}e_{\epsilon_i}(d_{\epsilon_i})+\int_{B_r(x_0)}|x|^{\frac12}|\tau_{\epsilon_i}|^2$$
holds for any $x_0\in\Omega$ and $0<r<{\rm{d}}(x_0,\partial\Omega)$,
it follows from (\ref{AMI3}) that there exist two constants $0<c, C<+\infty$
such that
\begin{equation}\label{AMI40}
(1-r^\frac12)\frac1{r}\int_{B_r(x_0)}e_{\epsilon_i}(d_{\epsilon_i})-cL_2^2 r^\frac12
\le (1+R^\frac12)\frac1{R}\int_{B_R(x_0)}e_{\epsilon_i}(d_{\epsilon_i})+CL_2^2R^\frac12
\end{equation}
holds for $x_0\in\Omega$ and $0<r<R<{\rm{d}}(x_0,\partial\Omega)$. Taking $i$ to $\infty$ in (\ref{AMI4}) yields
\begin{equation}\label{AMI5}
(1-r^\frac12)\frac1{r}\mu(B_r(x_0))-cL_2^2 r^\frac12
\le (1+R^\frac12)\frac1{R}\mu(B_R(x_0))+CL_2^2R^\frac12,
\ \forall\ x_0\in\Omega, \  0<r<R<{\rm{d}}(x_0,\partial\Omega).
\end{equation}
This implies that  for any $x\in\Omega$,
$$\Theta^1(\mu,x)=\lim_{r\rightarrow 0} r^{-1}\mu(B_r(x))$$
exists and is finite, and $\Theta^1(\mu,\cdot)$ is upper semicontinuous in $\Omega$.
Moreover, we have that there exists $C>0$ depending on $L_1$ such that
\begin{equation}\label{concentration1}
\Sigma=\Big\{x\in\Omega: \delta_0^2\le\Theta^1(\mu; x)\leq C\Big\}.
\end{equation}
Now we need the following general claims.

\smallskip
\noindent{\bf Claim 6.1}. $\Sigma\subset\Omega$ is a closed, $1$-dimensional rectifiable set with locally finite $H^1$-measure,
i.e.,
\begin{equation}\label{measure_est}
H^1(\Sigma\cap K)\le C(K, L_1, L_2)<+\infty, \ \forall\ K\subset\subset\Omega.
\end{equation}

Since $\Theta^1(\mu,\cdot)$ is upper semicontinuous, it follows from (\ref{concentration1}) that $\Sigma$ is closed. The $1$-rectifiability
of $\Sigma$ follows from the general rectifiability theorem by Preiss \cite{preiss} and Lin \cite{lin}. Let's sketch the measure
estimate (\ref{measure_est}). Since $\Sigma\cap K$ is compact,  it follows from the definition of $\Sigma$ and Vitali's covering lemma
that for any $\eta>0$ there exist a positive
integer $N=N_\eta$, $\displaystyle\{x_l\}_{l=1}^{N}\subset \Sigma\cap K$ and $\displaystyle\{r_l\}_{l=1}^N\subset\mathbb R_+$ such that
$\displaystyle\{B_{r_l}(x_l)\}_{l=1}^N$ are mutually disjoint, $\displaystyle\Sigma\subset\bigcup_{l=1}^N B_{5r_l}(x_l)$, and
$$\lim_{i\rightarrow\infty} \Phi_{r_l}(d_{\epsilon_i}; x_l)\ge \delta_0^2, \ 1\le l\le N.$$
This implies there exists a sufficiently large $i_0>1$ such that
$$\Phi_{r_l}(d_{\epsilon_{i_0}}; x_l)\ge \delta_0^2,  \ 1\le l\le N.$$
Hence
\begin{eqnarray*}
H^1_{5\eta}(\Sigma\cap K)\leq  5\sum_{l=1}^N r_l&\lesssim& \sum_{l=1}^N \int_{B_{r_l}(x_l)}
\left(e_{\epsilon_{i_0}}(d_{\epsilon_0})-\langle (x-x_l)\cdot\nabla d_{\epsilon_{i_0}}, \tau_{\epsilon_{i_0}}\rangle\right)
+|\tau_{\epsilon_i}|^2
\\
&\lesssim& \int_{\bigcup_{l=1}^N B_{r_l}(x_l)}\left(e_{\epsilon_{i_0}}(d_{\epsilon_0})-\langle (x-x_l)\cdot\nabla d_{\epsilon_{i_0}}, \tau_{\epsilon_{i_0}}\rangle\right)+|\tau_{\epsilon_i}|^2\\
&\leq& C(K)(L_1^2+L_2^2).
\end{eqnarray*}
Sending $\eta$ to zero, this implies (\ref{measure_est}).

\smallskip
\noindent{\bf Claim 6.2}. $d_\epsilon\rightarrow d_0$ in $H^1_{\rm{loc}}(\Omega\setminus\Sigma, \mathbb R^3)$,
$\nu\equiv 0$ in $\Omega\setminus \Sigma$, and $d_0\in C^{\frac12}
\cap W^{1,p}_{\rm{loc}}(\Omega\setminus\Sigma,\mathbb S^2)$ for all $2<p<3$.

In fact, for any $x_0\in\Omega\setminus\Sigma$, it follows from the definition of $\Sigma$ and (\ref{AMI3})
that there exist sufficiently small $\eta_0>0$, $r_0>0$, and sufficiently large $i_0\ge 1$ such that
$$ \Phi_{r_0}(d_{\epsilon_{i}}; x_0)\le \delta_0^2-\eta_0, \ \forall i\ge i_0.$$
This, combined with (\ref{AMI40}),  implies that
$$(1-r_0^\frac12)\frac1{r_0}\int_{B_{r_0}(x_0)}e_{\epsilon_i}(d_{\epsilon_i})\le \Phi_{r_0}(d_{\epsilon_i})+cL_2^2 r_0^\frac12
\le  \delta_0^2-\eta_0 +cL_2^2 r_0^\frac12\le\delta_0^2(1-r_0^\frac12), \ \forall\  i\ge i_0.$$
provide $r_0=r_0(\eta_0,\delta_0)>0$ is chosen to be sufficiently small. Hence
$$\frac1{r_0}\int_{B_{r_0}(x_0)}e_{\epsilon_i}(d_{\epsilon_i})\le\delta_0^2,\  \forall\  i\ge {i_0}.$$
Applying Lemma \ref{small_holder}, we conclude that $d_{\epsilon_i}\rightarrow d_0$ in
$H^1(B_{\frac{r_0}2}(x_0), \mathbb R^3)$ and $\displaystyle\frac{(1-|d_{\epsilon_i}|^2)^2}{\epsilon_i^2}
\rightarrow 0$ in $L^1(B_{\frac{r_0}2}(x_0))$. Hence $\nu\equiv 0$ in $B_{\frac{r_0}2}(x_0)$. It also
follows from (\ref{lp-bound}) and (\ref{holder_est}) that
$d_0\in C^\frac12\cap W^{1,p}(B_{\frac{r_0}2}(x_0))$ for all $2<p<3$.  Since $x_0\in\Omega\setminus\Sigma$
is arbitrary, Claim 6.2 follows.

Denote the singular set of $d_0$ by
$${\rm{sing}}(d_0):=\Big\{x\in\Omega: d_0 \ {\rm{is\ discontinuous\ at}}\ x\Big\}.$$
Then we have

\smallskip
\noindent{\bf Claim 6.3}.
$\Sigma={\rm{supp}}(\nu)\cup {\rm{sing}}(d_0)$.

It is easy to see from claim 6.2 that ${\rm{supp}}(\nu)\cup {\rm{sing}}(d_0)\subset\Sigma$. If $x_0\notin {\rm{supp}}(\nu)\cup {\rm{sing}}(d_0)$, then there exists $r_0>0$ such that $\nu(B_{r_0}(x_0))=0$ and $d_0\in C(B_{r_0}(x_0))$. By Lemma \ref{small_holder}, we
have that $d_0\in H^1(\Omega,\mathbb S^2)$ is an approximated harmonic map with tension field $\tau_0$, i.e.,
\begin{equation}\label{AHM}
\Delta d_0+|\nabla d_0|^2 d_0=\tau_0.
\end{equation}
For small $\epsilon>0$, assume $r_0>0$ such that
$${\rm{osc}}_{B_{r_0}(x_0)}(d_0)\le\epsilon.$$
Then by the standard hole filling argument (see \cite{LW4}), there exists $\theta_0\in (0, \frac12)$ such that
$$\frac1{\theta_0 r_0}\int_{B_{\theta_0 r_0}(x_0)}|\nabla d_0|^2\le
\frac1{2r_0}\int_{B_{r_0}(x_0)}|\nabla d_0|^2+Cr_0.$$
Iterating this inequality then implies that there exists $\alpha_0\in (0,\frac12)$ such that for any $0<r\le r_0$,
$$\frac1{r}\int_{B_r(x_0)}|\nabla d_0|^2\le \big(\frac{r}{r_0}\big)^{2\alpha_0}
\frac1{r_0}\int_{B_{r_0}(x_0)}|\nabla d_0|^2+Cr.$$
In particular, we have
$$\lim_{r\rightarrow 0} \frac1{r}\int_{B_r(x_0)}|\nabla d_0|^2=0,$$
so that $\Theta^1(\mu,x_0)=0$ and hence $x_0\notin\Sigma$. This proves claim 6.3.

\smallskip
\noindent{\bf Claim 6.4}. For $H^1$ a.e. $x\in\Sigma$,
$$\displaystyle\Theta^1(\nu;x)
=\lim_{r\rightarrow 0}\frac1{r}\nu(B_r(x))$$
exists and $\delta_0^2\le\Theta^1(\nu;x)\le C$, and $\nu=\Theta^1(\nu, \cdot)H^1 {\rm{L}}\Sigma$.

Since $d_0\in H^1(\Omega,\mathbb S^2)$, by Federer-Ziemmer's theorem \cite{FZ}
that for $H^{1}$ a.e. $x\in\Omega$
\begin{equation}\label{d-density}
\Theta^1\big(|\nabla d_0|^2,x\big):=\lim_{r\rightarrow 0}\frac1{r}\int_{B_r(x)}|\nabla d_0|^2=0.
\end{equation}
The conclusions of claim 6.4 then follow from this. See \cite{lin} or \cite{LW4} for the detail.
We have not used the condition (i) in the definition of ${\bf X}(L_1,L_2,a;\Omega)$ during the proof of the above claims.
Next we employ this condition to show that $\nu\equiv 0$ in $\Omega$. More precisely, we have

\smallskip
\noindent{\bf Claim 6.5}. $H^1(\Sigma)=0$ and $\nu\equiv 0$, $d_{\epsilon}\rightarrow d_0$ in $H^1_{\rm{loc}}(\Omega,\mathbb R^3)$.

Suppose that $H^1(\Sigma)>0$. Then, as in \cite{lin} and \cite{LW2}, since $\Theta^1(\nu, \cdot)$ is $H^1$-measurable, it is approximately continuous
for $H^1$ a.e. $x\in\Sigma$. This, combined with (\ref{d-density}) and the $1$-rectifiability of $\Sigma$, implies
that there exists $x_0\in\Sigma$ such that
\begin{itemize}
\item [i)] $\Theta^1(\nu,\cdot)$ is $H^1$-approximately continuous at $x_0$, and $\delta_0^2\le\Theta^1(\nu,x_0)\le C$.
\item [ii)] $\Sigma$ has  $1$-dimensional tangent plane $T_{x_0}\Sigma$ at $x_0$.
\item [iii)] $\Theta^1(|\nabla d_0|^2,x_0)=0$.
\end{itemize}
For simplicity, assume $x_0=0\in\Sigma$ and $T_{x_0}\Sigma=\big\{(0,0, x_3): x_3\in\mathbb R\big\}.$
For $r_i\rightarrow 0$, define $\widetilde {d}_i(x)=d_{\epsilon_i}(r_i x)$ and $\widetilde \tau_i(x)=r_i^2\tau_{\epsilon_i}(r_i x)$
for $x\in \Omega_i\equiv r_i^{-1}\Omega$, and $\widetilde{\epsilon}_i=\frac{\epsilon_i}{r_i}$. Then we have
\begin{equation}\label{AGL5}
\Delta \widetilde d_i+\frac{1}{\widetilde\epsilon_i^2}(1-|\widetilde d_i|^2)\widetilde d_i=\widetilde\tau_i
\ \ \ {\rm{in}}\ \Omega_i .
\end{equation}
By following the blow-up scheme outlined in \cite{LW1}, we can assume that after passing to possible subsequences, there
exists a tangent measure $\mu_*$ of $\mu$ at $0$ such that
$$e_{\widetilde\epsilon_i}(\widetilde d_i)\,dx\rightharpoonup \mu_*,$$
as convergence of Radon measures on $\mathbb R^3$,
and
$$\widetilde d_i \rightharpoonup {\rm{constant\ \ in\ }}
H^1(\mathbb R^3,\mathbb R^3).$$
Moreover,
\begin{equation}
\label{tangent}
\mu_*=\Theta^1(\nu, 0)H^1{\rm{L}}\big\{(0,0, x_3): \ x_3\in\mathbb R\big\}.
\end{equation}
Since $\widetilde d_i$ is a solution of the equation (\ref{AGL5}), $\widetilde d_i$ also satisfies the energy monotonicity
formula (\ref{AMI1}). In particular, we have
\begin{equation}\label{AMI4}
\Phi_R(\widetilde d_i; x)\ge \Phi_r(\widetilde d_i; x)+\frac12\int_{B_R(x)\setminus
B_r(x)}|y-x|^{-1}\Big(\big|\frac{\partial \widetilde d_i}{\partial |y-x|}\big|^2+\frac{(|1-|\widetilde d_i|^2)^2}{\widetilde\epsilon_i^2}\Big),
\end{equation}
for $x\in \Omega_i$ and $0<r\le R<{\rm{d}}(x,\partial\Omega_i)$, where
$$\Phi_r(\widetilde d_{i}; x)
:=\frac1{r}\int_{B_r(x)}\big(e_{\widetilde\epsilon_i}(\widetilde d_i)
-\langle (y-x)\cdot \nabla \widetilde d_i,\widetilde\tau_{i}\rangle\big)
+\frac12\int_{B_r(x)} |y-x||\widetilde\tau_i|^2,
\ x\in\Omega_i, \ 0<r<{\rm{d}}(x,\partial\Omega_i).$$
Since
$$\Big|r^{-1}\int_{B_r(x)}\langle (y-x)\cdot \nabla \widetilde d_i,\widetilde\tau_{i}\rangle\Big|
\le Cr_i^\frac12\ {\rm{and}}\ \int_{B_r(x)} |y-x||\widetilde\tau_i|^2\le Cr_i,
$$
it follows that
\begin{equation}\label{limit_density}
\lim_{i\rightarrow \infty}\Phi_r(\widetilde d_i,x)=\frac1{r}\mu_*(B_r(x)),
\end{equation}
for $x\in\mathbb R^3$ and $r>0$. It is clear that (\ref{limit_density}), (\ref{tangent}), and (\ref{AMI4}) imply
that
\begin{equation}\label{radial1}
\lim_{i\rightarrow\infty} \int_{B_R(x)\setminus
B_r(x)}|y-x|^{-1}\Big(\big|\frac{\partial \widetilde d_i}{\partial |y-x|}\big|^2
+\frac{(|1-|\widetilde d_i|^2)^2}{\widetilde\epsilon_i^2}\Big)=0,
\end{equation}
holds for any $x=(0,0,x_3)\in T_0\Sigma$ and $0<r<R$.
Applying (\ref{radial1}) to two center points $(0, 0,0)$ and $(0,0, 2)$, we can obtain
\begin{equation}\label{radial2}
\lim_{i\rightarrow\infty} \int_{B_1^2\times [-1,1]}
\Big(\big|\frac{\partial \widetilde d_i}{\partial x_3}\big|^2
+\frac{(|1-|\widetilde d_i|^2)^2}{\widetilde\epsilon_i^2}\Big)=0.
\end{equation}
Recall that the condition (i) of ${\bf X}(L_1,L_2,a; \Omega)$ implies
\begin{equation}\label{range}
|\widetilde d_i|\le 1, \ \widetilde d_i^3\ge -1+a.
\end{equation}

Now we indicate how to produce a nontrivial harmonic map $\omega:\mathbb R^2\to\mathbb S^2$ with finite energy.
Define
$f_i:[-1,1]\to\mathbb R_+$ by
$$f_i(t)=\int_{B_1^2}\Big(\big|\frac{\partial \widetilde d_i}{\partial x_3}\big|^2
+\frac{(|1-|\widetilde d_i|^2)^2}{\widetilde\epsilon_i^2}\Big)(x,t)\,dx,$$
$g_i:[-1,1]\to\mathbb R_+$ by
$$g_i(t)=\int_{B_1^2} e_{\widetilde\epsilon_i}(\widetilde d_i)(x,t)\,dx,$$
and $h_i:[-1, 1]\to\mathbb R_+$ by
$$h_i(t)=\int_{B_1^2}|\widetilde\tau_i|^2(x,t)\,dx.$$
By Fubini's theorem and (\ref{radial2}), we have
$$\lim_{i\rightarrow \infty}\big\|f_i\big\|_{L^1([-1,1])}=0 \ {\rm{and}}\ \limsup_{i\rightarrow \infty}\big\|h_i\big\|_{L^1([-1,1])}
\le L_2.$$
Thus by the weak $L^1$-estimate of the Hardy-Littlewood maximal function we have
that for any $\beta>0$ there exists a set $E_\beta\subset [-\frac12,\frac12]$, with $|E_\beta|\ge 1-\beta$,
such that
\begin{equation}\label{maximal1}
\lim_{i\rightarrow\infty} \sup_{0<r\le \frac12} \frac1{r}\int_{t-r}^{t+r}f_i(x_3)\,dx_3=0,
\ \forall \ t\in E_\beta,
\end{equation}
and
\begin{equation}\label{maximal2}
\lim_{i\rightarrow\infty} \sup_{0<r\le \frac12} \frac1{r}\int_{t-r}^{t+r}h_i(x_3)\,dx_3\le CL_2,
\ \forall \ t\in E_\beta.
\end{equation}
We can also assume that there exists $F_\beta\subset E_\beta$, with $|F_\beta|\ge 1-2\beta$, such that
\begin{equation}\label{slice}
\lim_{i\rightarrow\infty}g_i(t)=\Theta^1(\nu,0), \ \forall\ t\in F_\beta.
 \end{equation}
For simplicity, assume $0\in F_\beta$. For $\delta_0>0$ given by Lemma \ref{small_holder},
there exist $\{x_i\}(\subset B_1^2)\rightarrow (0,0)$ and $\lambda_i\rightarrow 0^+$ such that
\begin{equation}\label{concentration3}
\int_{B^2_{\lambda_i}(x_i)}e_{\widetilde \epsilon_i}(\widetilde d_i)(x,0)\,dx
=\frac{\delta_0^2}{C(3)}=\max_{z\in B_\frac12^2}
\int_{B^2_{\lambda_i}(z)}e_{\widetilde \epsilon_i}(\widetilde d_i)(x,0)\,dx,
\end{equation}
where $C(3)>0$ is a large constant to be chosen later.

Define the rescaling maps
$$\widehat{d}_i(x,x_3)=\widetilde d_i((x_i,0)+\lambda_i (x,x_3)), \ (x,x_3)\in
\widehat\Omega_i:=\lambda_i^{-1}\big(\Omega_i\setminus\{(x_i,0)\}\big).$$
Then $\widehat {d}_i$ solves
\begin{equation}
\Delta\widehat{d}_i+\frac{1}{\widehat\epsilon_i^2}(1-|\widehat{d}_i|^2)\widehat d_i=\widehat \tau_i
\ \ {\rm{in}}\ \ \widehat\Omega_i,
\end{equation}
where $\widehat\epsilon_i=\frac{\widetilde\epsilon_i}{\lambda_i}$ and
$\widehat \tau_i(x, x_3)=\lambda_i^2\widetilde\tau_i(\lambda_i x,\lambda_i x_3)$.
It follows from (\ref{maximal1}), (\ref{maximal2}),  and (\ref{concentration3}) that
\begin{equation}\label{maximal3}
\lim_{i\rightarrow\infty} \sup_{0<r\le \lambda_i^{-1}} \frac1{r}
\int_{-r}^{r}\int_{B^2_{\lambda_i^{-1}}(0)} \Big(\Big|\frac{\partial \widehat  d_i}{\partial x_3}\Big|^2
+\frac{(|1-|\widehat d_i|^2)^2}{\widehat\epsilon_i^2}\Big)(x,t)\,dxdt=0,
\end{equation}
\begin{equation}\label{maximal4}
\sup_{0<r\le \lambda_i^{-1}} \frac1{r}
\int_{-r}^{r}\int_{B^2_{\lambda_i^{-1}}(0)} |\widehat\tau_i|^2(x,t)\,dxdt=
\lambda_i^2 \frac1{\lambda_ir}\int_{-\lambda_i r}^{\lambda_i r}\int_{B_1^2}|\widetilde\tau_i|^2(x,t)\,dxdt
\le CL_2^2\lambda_i^2\rightarrow 0,
\end{equation}
and
\begin{equation}\label{concentration4}
\int_{B^2_1(0)}e_{\widehat \epsilon_i}(\widehat d_i)(x,0)\,dx
=\frac{\delta_0^2}{C(3)}=\max_{z\in B^2_{\lambda_i^{-1}}}
\int_{B^2_{1}(z)}e_{\widehat \epsilon_i}(\widehat d_i)(x,0)\,dx,
\end{equation}
Moreover, for $\phi\in C_0^\infty(B_2^2(0))$, direct calculations imply
that
\begin{eqnarray}\label{uniform_third}
&&\frac{\partial}{\partial x_3}\int_{\mathbb R^2}
\phi^2(x)e_{\widehat\epsilon_i}(\widehat d_i)(x,x_3)\,dx\nonumber\\
&=&-\int_{\mathbb R^2}\big\langle\nabla_x \widehat d_i, \frac{\partial\widehat d_i}{\partial x_3}\big\rangle\cdot\nabla_x(\phi^2) +
\frac{\partial}{\partial x_3}\int_{\mathbb R^2}\big|\frac{\partial \widehat d_i}{\partial x_3}\big|^2\phi^2
-\int_{\mathbb R^2}\big\langle\widehat \tau_i, \frac{\partial\widehat d_i}{\partial x_3}\big\rangle\phi^2.
\end{eqnarray}
It follows from (\ref{concentration4}, (\ref{uniform_third}), (\ref{maximal3}), and (\ref{maximal4})
that
\begin{equation}\label{uniform_small}
\frac{1}{2}\int_{-2}^2\int_{B_2^2(x,0)}e_{\widehat \epsilon_i}(\widehat d_i)(x,x_3)\,dxdx_3\le C\frac{\delta_0^2}{C(3)}\le
\delta_0^2, \ \forall x\in B^2_{\lambda_i^{-1}},
\end{equation}
provided $C(3)>0$ is chosen to be sufficiently large. Hence Lemma \ref{small_holder}
implies that there exists $\widehat d\in H^1_{\rm{loc}}(\mathbb R^2\times [-2,2], \mathbb R^3)$
such that
$$\widehat d_i\rightarrow \widehat d\ \ \ {\rm{in}}\ \ H^1_{\rm{loc}}(\mathbb R^2\times [-2,2]).
$$
It follows from (\ref{maximal3}) that $\displaystyle
|\widehat d|=1 \ {\rm{and}}\ \frac{\partial \widehat d}{\partial x_3}=0\ {\rm{a.e.}}\ \mathbb R^2\times [-2,2]$
so that $\widehat d(x,x_3)=\widehat d(x)$ is independent of $x_3$. By
(\ref{concentration4}) and (\ref{slice}), we have that
\begin{equation}
\label{nontrivial}
\frac{\delta_0^2}{C(3)}\le \frac12\int_{\mathbb R^2}|\nabla\widehat d|^2\le \Theta(\nu,0).
\end{equation}
Hence $\widehat d\in \dot{H}^1(\mathbb R^2,\mathbb S^2)$.
Moreover, it follows from (\ref{maximal4}) and (\ref{nontrivial})  that $\widehat h$ is a nontrivial smooth harmonic map
from $\mathbb R^2$ to $\mathbb S^2$ with finite energy, i.e.,
$$\Delta\widehat d+\big|\nabla\widehat d\big|^2\widehat d=0 \ {\rm{in}}\ \mathbb R^2.$$
On the other hand, it follows from (\ref{range}) that
$\displaystyle \widehat d_i^3(x,x_3)\ge -1+a$ for $(x,x_3)\in\mathbb R^2\times [-2,2]$.
Hence we have
$$\widehat d^3(x)\ge -1+a, \ x\in\mathbb R^2.$$
In particular, ${\rm{deg}}(\widehat d)=0$. Since any nontrivial harmonic map from $\mathbb R^2$ to $\mathbb S^2$
with finite energy has non-zero degree, we conclude that $\widehat d=$ constant. This yields the desired contradiction.
Hence the conclusion of claim 6.5 holds true.

\smallskip
\noindent{\bf Claim 6.6}. $\Sigma=\emptyset$, and $d_0\in W^{2,2}_{\rm{loc}}(\Omega,\mathbb S^2_{-1+a})$.

Suppose $\Sigma\not=\emptyset$ and $x_0\in\Sigma$. By the definition of $\Sigma$, we have
that
\begin{equation}
\label{concentration5}
\lim_{i\rightarrow \infty}\Phi_r(d_{\epsilon_i}, x_0)\ge \delta_0^2,\ \forall\  r>0.
\end{equation}
It follows from claim 6.5
$$e_{\epsilon_i}(d_{\epsilon_i})\,dx\rightharpoonup \frac12|\nabla d_0|^2\,dx$$
as convergence of Radon measures as $i\rightarrow\infty$. In particular, $d_{\epsilon_i}\rightarrow d_0$
in $H^1_{\rm{loc}}(\Omega, \mathbb R^3)$. Assume $\tau_{\epsilon_i}\rightharpoonup
\tau_0$ in $L^2(\Omega,\mathbb R^3)$. Then we know that $d_0$ is an approximated harmonic
map to $\mathbb S^2_{-1+a}$ with tension field $\tau_0$:
\begin{equation}\label{AHM0}
\Delta d_0+|\nabla d_0|^2 d_0=\tau_0 \ \ {\rm{in}}\ \ \Omega,
\end{equation}
and $d_0$ satisfies the energy monotonicity formula:
\begin{equation}\label{AMI7}
\Psi_R(d_0; x)\ge \Psi_r(d_0; x)+\frac12\int_{B_R(x)\setminus
B_r(x)}|y-x|^{-1}\big|\frac{\partial d_0}{\partial |y-x|}\big|^2,
\end{equation}
for $x\in \Omega$ and $0<r\le R<{\rm{d}}(x,\partial\Omega)$, where
\begin{equation}\label{energy_density}
\Psi_r(d_0; x)
:=\frac1{r}\int_{B_r(x)}\big(\frac12|\nabla d_0|^2
-\langle (y-x)\cdot \nabla d_0, \tau\rangle\big)
+\frac12\int_{B_r(x)} |y-x||\tau_0|^2.
\end{equation}
It follows from (\ref{concentration5}) and (\ref{AMI7})
that
\begin{equation}\label{concentration6}
\Psi(d_0,x_0):=\lim_{R\downarrow 0}\Psi_R(d_0,x_0)\ge \delta_0^2.
\end{equation}
For $r_i\rightarrow 0$, define the blow-up sequence of $d_0$ at $x_0$,
 $d_i(x)=d_0(x_0+r_ix): B_2\to\mathbb S^2_{-1+a}$. Then we have
$$\lim_{i\rightarrow\infty}\frac12\int_{B_1}|\nabla d_i|^2
=\lim_{i\rightarrow\infty}\Psi_{r_i}(d_0, x_0)\ge \delta_0^2.$$
It is clear that $d_i$ is an approximated harmonic map with tension
field $\tau_i(x)=r_i^2\tau_0(r_i x)$ such that
\begin{itemize}
\item[i)] $d_i(B_1)\subset \mathbb S^2_{-1+a}$.
\item[ii)] $\displaystyle E(d_i)=\frac12\int_{B_1}|\nabla d_i|^2\le C_1.$
\item[iii)] $d_i$ satisfies the energy monotonicity inequality (\ref{AMI7}), with $d_0$ and $\tau_0$ replaced by
$d_i$ and $\tau_i$.
\item[iv)] $\|\tau_i\|_{L^2(B_1)}\le C\sqrt{r_i}.$
\end{itemize}
Hence $\{d_i\}\subset {\bf Y}(C_1, C\sqrt{r_i}, a; B_1)$. It follows from Theorem 7.1 below that there exists a harmonic map
$\omega\in H^1(B_1,\mathbb S^2_{-1+a})$ such that
$d_i\rightarrow \omega$ in $H^1_{\rm{loc}}(B_1,\mathbb R^3)$ so that
$$\frac12\int_{B_\frac12}|\nabla\omega|^2\ge\delta_0^2.$$
Moreover, (\ref{AMI7}) implies that
$$\int_{B_\frac12}\big|\frac{\partial\omega}{\partial |x|}\big|^2=0,$$
so that $\omega(x)=\omega(\frac{x}{|x|})$ is homogeneous of degree zero
and $\omega:\mathbb S^2\to\mathbb S^2_{-1+a}$ is a nontrivial harmonic map.
This is impossible. Hence $\Sigma=\emptyset$ and hence
Lemma \ref{small_holder1} implies $d_0\in W^{2,2}_{\rm{loc}}(\Omega,\mathbb S^2)$.
The proof is complete.
\end{proof}

\section{$H^1$ precompactness of suitable approximated harmonic map to $\mathbb S^2$}

For $0<a\le 2$ and $L_1, L_2>0$, define the set ${\bf Y}(L_1, L_2,a;\Omega)$ consisting of maps
$d\in H^1(\Omega,\mathbb S^2)$ that are suitable approximated harmonic maps, i.e.,
\begin{equation}\label{AHM3}
\Delta d+|\nabla d|^2 d=\tau\ \ {\rm{in}}\ \ \Omega
\end{equation}
that satisfy, in addition to (\ref{HMI2}), the following properties:
\begin{itemize}
\item [(i)] $d^3(x)\ge -1+a  \ {\rm{for\ a.e.}}\  x\in\Omega$.
\item [(ii)] $\displaystyle E(d)=\frac12\int_\Omega |\nabla d|^2\,dx\le L_1$.
\item [(iii)] $\displaystyle\big\|\tau\big\|_{L^2(\Omega)}\le L_2.$
\end{itemize}

\begin{thm}\label{precomp2} For any $a\in (0,2],\ L_1>0,$ and $L_2>0$, the set $\mathbf{Y}(L_1, L_2,a;\Omega)$
is precompact in
$H^1_{\rm{loc}}(\Omega,\mathbb S^2)$. In particular, if
$\{d_i\}\subset {\bf Y}(L_1, L_2, a;\Omega)$
is a sequence of approximated harmonic maps, with tension fields $\{\tau_i\}$,  then there exist
$\tau_0\in L^2(\Omega,\mathbb R^3)$ and
an approximated harmonic map $d_0\in {\bf Y}(L_1, L_2, a;\Omega)$, with tension field $\tau_0$,
such that after passing to possible subsequences,
$d_i\rightarrow d_0$ in $H^1_{\rm{loc}}(\Omega, \mathbb S^2)$ and $\tau_i\rightharpoonup
\tau_0$ in $L^2(\Omega,\mathbb R^3)$.
Moreover, $d_0\in W^{2,2}_{\rm{loc}}(\Omega, \mathbb S^2)$.
\end{thm}
\begin{proof} The proof is almost identical to that of Theorem \ref{precomp1}. Here we only sketch it.
Suppose that $d_i\rightharpoonup d_0$ in $H^1(\Omega)$, but not strongly in $H^1_{\rm{loc}}(\Omega)$.
Then there exists a Radon measure $\nu\ge 0$ ($\nu\not\equiv 0$) such that
$$\frac12|\nabla d_i|^2\,dx\rightharpoonup \mu:=\frac12|\nabla d_0|^2\,dx+\nu$$
as convergence of Radon measures as $i\rightarrow\infty$.
Define the concentration set
$$\Sigma=\bigcap_{0<r<{\rm{d}}(x,\partial\Omega)}\Big\{x\in\Omega: \ \liminf_{i\rightarrow\infty}\Psi_r(d_i, x)
\ge\delta_0^2\Big\},$$
where $\Psi_r(d_i,x)$ is given by (\ref{energy_density}) and $\delta_0>0$ is the constant given by
Lemma \ref{small_holder1}.

Since $\nu(\Omega)>0$,
with the help of Lemma \ref{HMI1} and Lemma \ref{small_holder1}, the same argument as in Theorem
\ref{precomp1} yields
\begin{itemize}
\item[(i)] $\Sigma$ is a $1$-dimensional rectifiable, closed set, with $H^1(\Sigma)>0$.
\item[(ii)] there exists $C>0$ depending on $L_1$ and $L_2$ such that
$$\Sigma=\Big\{x\in\Omega: \ \delta_0^2\le\Theta^1(\mu,x)\le C\Big\},$$
where $\displaystyle\Theta^1(\mu,x)=\lim_{r\rightarrow 0}\Theta^1_r(\mu,x)\big(=\lim_{r\rightarrow 0}\frac{1}{r}\mu(B_r(x))\big)$
is the $1$-dimensional density of $\mu$ at $x$.
\item[(iii)] $\displaystyle\Sigma={\rm{supp}}(\nu)\cup{\rm{sing}}(d_0)$.
\item[(iv)] For $H^1$ a.e. $x\in\Sigma$, $$\displaystyle\Theta^1\big(|\nabla d_0|^2\,dx, x\big)
=\lim_{r\rightarrow 0}\frac1{r}\int_{B_r(x_0)}|\nabla d_0|^2=0,$$
and $\displaystyle\Theta^1(\nu,x)=\lim_{r\rightarrow 0}\frac{1}{r}\nu\big(B_r(x)\big)$ exists and equals
to $\Theta^1(\mu,x)$.
\end{itemize}
As in claim 6.5, we can choose a generic point $x_0\in\Sigma$ such that
\begin{itemize}
\item[a)] $\displaystyle\Theta^1(|\nabla d_0|^2\,dx, x_0)=0$.
\item[b)] $\Theta^1(\nu,\cdot)$ is $H^1$-approximately continuous at $x_0$ and
$\delta_0^2\le\Theta^1(\nu, x_0)\le C$.
\item[c)] $\Sigma$ has $1$-dimensional tangent plane $T_{x_0}\Sigma$ at $x_0$.
\end{itemize}
Then we perform the blow-up procedure of $d_i$ at $x_0$  exactly as what we did in claim 6.5 (we leave the detail to interested
readers). As a consequence, we will obtain a harmonic map $\omega: \mathbb R^2\to\mathbb S^2$ such
that
$$0<\int_{\mathbb R^2}|\nabla\omega|^2<+\infty, \ \omega^3(x)\ge -1+a, \ \forall \ x\in\mathbb R^2.$$
This is impossible. Hence $\nu\equiv 0$ and $d_i\rightarrow d_0$ in $H^1_{\rm{loc}}(\Omega)$. Since
${\bf Y}(L_1, L_2, a;\Omega)$ is closed under $H^1_{\rm{loc}}(\Omega)$-convergence, we conclude
that $d_0\in {\bf Y}(L_1, L_2, a; \Omega)$.

Now we want to show that
\begin{equation}\label{zero_density}
\Theta^1\big(|\nabla d_0|^2\,dx, x\big)=0, \ \forall\ x\in\Omega.
\end{equation}
For, otherwise,
there exists $x_0\in\Omega$ and $\lambda_i\rightarrow 0$ such that
$\widetilde d_i(x)=d_0(x_0+\lambda_i x): B_2\to \mathbb S^2_{-1+a}$
satisfies
$$\frac12\int_{B_2}|\nabla\widetilde d_i|^2=\frac1{2\lambda_i}\int_{B_{2\lambda_i}(x_0)}|\nabla d_0|^2
\rightarrow\Theta^1\big(|\nabla d_0|^2\,dx, x_0\big)>0 \ \ {\rm{as}}\ \ i\rightarrow\infty.$$
It is easy to see that $\{\widetilde d_i\}\subset {\bf Y}(L_1, L_2\sqrt{\lambda_i}, a; B_2)$. The compactness of ${\bf Y}(L_1, L_2\sqrt{\lambda_i}, a;
B_2)$
implies that there exists a harmonic map $\widetilde d\in H^1(B_2, \mathbb S^2)$, with $\widetilde d^3(x)\ge -1+a$
for $x\in B_2$,  such that $\widetilde d_i\rightarrow \widetilde d$ in $H^1_{\rm{loc}}(B_2)$. Moreover, it follows from the monotonicity
inequality (\ref{HMI2}) for $\widetilde d_i$ that
$$\int_{B_\frac32}\big|\frac{\partial\widetilde d}{\partial |x|}\big|^2=0.$$
Hence $\widetilde d(x)=\widetilde d\big(\frac{x}{|x|}\big): \mathbb S^2\to \mathbb S^2_{-1+a}$ is a nontrivial harmonic map, which is impossible. This proves (\ref{zero_density}) and hence Lemma \ref{small_holder1}
yields $d_0\in W^{2,2}_{\rm{loc}}(\Omega,\mathbb S^2)$.
\end{proof}

\medskip
\section{Global weak solutions of (\ref{LLF}) and proofs of Theorem \ref{existence} and Theorem \ref{compactness0}}

In this section, we will utilize the existence of global solutions to
the Ginzburg-Landau approximation (\ref{MLLF}) of the nematic liquid crystal flow (\ref{LLF}) and the compactness
Theorem \ref{precomp1} to show the existence of global weak solutions to (\ref{LLF}).

For $\epsilon>0$, consider the modified Ginzburg-Landau approximation of (\ref{LLF}):
\begin{equation}\label{MLLF}
\begin{cases}
\partial_t u+u\cdot\nabla u-\Delta u+\nabla P=-\nabla\cdot(\nabla d\odot \nabla d) & {\rm{in}}\ \ \Omega\times (0,+\infty),\\
\nabla\cdot u = 0 & {\rm{in}}\ \ \Omega\times (0,+\infty),\\
\partial_t d+u\cdot\nabla d=\Delta d+\frac{1}{\epsilon^2}(1-|d|^2) d & {\rm{in}}\ \ \Omega\times (0,+\infty).
\end{cases}
\end{equation}
First, we have the following result on the existence of global solutions to (\ref{MLLF}).
\begin{thm} \label{approx_solution}
For any $\epsilon>0$, $u_0\in {\bf H}$, and $d_0\in H^1(\Omega,\mathbb S^2)$,
there exists a global weak solution $(u_\epsilon,d_\epsilon):\Omega\times [0,+\infty)\to\mathbb R^3\times \mathbb R^3$
of the equation (\ref{MLLF}) under the initial and boundary condition (\ref{IBC}) that satisfies
\begin{itemize}
\item[(i)] $|d_\epsilon|\le 1$ a.e. $(x,t)\in\Omega\times [0,+\infty)$.
\item[(ii)] the global energy inequality: there exists a measure zero set $E\subset (0,+\infty)$ such that
for any $0\le t_1, t_2\in \mathbb R\setminus E$ with $t_1<t_2$,
\begin{eqnarray}\label{global_energy_ineq1}
&&\int_\Omega \big(|u_\epsilon|^2+|\nabla d_\epsilon|^2+\frac{1}{2\epsilon^2}(1-|d_\epsilon|^2)^2\big)(t_2)
+2\int_{t_1}^{t_2}\int_\Omega \Big(|\nabla u_\epsilon|^2+\big|\Delta d_\epsilon+\frac{1}{\epsilon^2}(1-|d_\epsilon|^2)d_\epsilon\big|^2\Big)\nonumber\\
&&\le \int_\Omega \big(|u_\epsilon|^2+|\nabla d_\epsilon|^2+\frac{1}{2\epsilon^2}(1-|d_\epsilon|^2)^2\big)(t_1).
\end{eqnarray}
\item [(iii)] If, in addition, $d_0^3(x)\ge 0$ a.e. $x\in\Omega$, then $d^3_\epsilon(x,t)\ge 0$ a.e. $(x,t)\in\Omega\times (0,+\infty)$.
\end{itemize}
\end{thm}
\begin{proof} The existence is based on the Galerkin method and the energy method. The reader can refer to
the proof presented by \cite{LL1} \S2.
The properties (i) and (iii) follow from Lemma \ref{MP1} and Lemma \ref{MP2}.
\end{proof}
Now we would like to study the convergence of sequence of solutions $(u_\epsilon, d_\epsilon)$ constructed by
Theorem \ref{approx_solution} as $\epsilon$ tends to zero.

\medskip
\noindent{\bf Proof of Theorem \ref{existence}}. Since $|d_0|=1$ and $d_0^3\ge 0$, it follows from Theorem \ref{approx_solution}
that for $\epsilon>0$,  there exists global weak solutions $(u_\epsilon, d_\epsilon):\Omega\times [0,+\infty)\to\mathbb R^3\times\mathbb R^3$ of (\ref{MLLF}) that satisfies all the three properties (i), (ii), and (iii).
It follows from (\ref{global_energy_ineq1}) that
\begin{eqnarray}\label{global_energy_bound}
&&\sup_{\epsilon>0}\sup_{0<t<+\infty}\int_\Omega \Big[|u_\epsilon|^2+|\nabla d_\epsilon|^2+\frac{1}{2\epsilon^2}(1-|d_\epsilon|^2)^2\Big](t)+2\int_0^{\infty}\int_\Omega
\Big[|\nabla u_\epsilon|^2+\big|\Delta d_\epsilon+\frac{1}{\epsilon^2}(1-|d_\epsilon|^2)d_\epsilon\big|^2\Big]\nonumber\\
&&\le \int_\Omega \big(|u_0|^2+|\nabla d_0|^2\big):=C_0
\end{eqnarray}
It is not hard to see from the equation (\ref{MLLF}) and the  energy inequality (\ref{global_energy_bound}) that
there exists $p>3$ such that
for any $0<T<+\infty$,
\begin{eqnarray}\label{t-uniform}
\sup_{\epsilon>0}\Big[\|\partial_t u_\epsilon\|_{L^\frac54(\Omega\times [0,T])+L^2([0,T], H^{-1}(\Omega))+L^2([0,T], W^{-1,p}(\Omega))}+\|\partial_t d_\epsilon\|_{L^\frac54(\Omega\times [0,T])}\Big]<+\infty.
\end{eqnarray}
Hence, by Aubin-Lions' Lemma \cite{temam} we have that after taking possible subsequences,  there exists $u\in L^\infty_tL^2_x\cap L^2_tH^1_x(\Omega\times
\mathbb R_+,\mathbb R^3)$ and $d\in L^\infty_tH^1_x(\Omega\times\mathbb R_+, \mathbb S^2)$ such that
\begin{equation}\label{weak_strong}
(u_\epsilon, d_\epsilon)\rightarrow (u, d) \ {\rm{in}}\ L^2_{\rm{loc}}(\Omega\times \mathbb R_+),
\ (\nabla u_\epsilon,\nabla d_\epsilon)\rightharpoonup (\nabla u, \nabla d) \ {\rm{in}}\ L^2_tL^2_x(\Omega\times \mathbb R_+).
\end{equation}
It follows from (\ref{global_energy_bound}) and Fatou's lemma that
\begin{equation}\label{fatou}
\int_0^{\infty}\liminf_{\epsilon\rightarrow 0}\int_\Omega
\Big|\Delta d_\epsilon+\frac{1}{\epsilon^2}(1-|d_\epsilon|^2)d_\epsilon\Big|^2
\le C_0.
\end{equation}
For $\Lambda>>1$, define
$$G_\Lambda^T:=\Big\{t\in [0,T]:
\ \liminf_{\epsilon\rightarrow 0}\int_\Omega
\Big|\Delta d_\epsilon+\frac{1}{\epsilon^2}(1-|d_\epsilon|^2)d_\epsilon\Big|^2(t)\le \Lambda\Big\},$$
and
$$B_\Lambda^T:=[0,T]\setminus G_\Lambda^T=\Big\{t\in [0,T]:
\ \liminf_{\epsilon\rightarrow 0}\int_\Omega
\Big|\Delta d_\epsilon+\frac{1}{\epsilon^2}(1-|d_\epsilon|^2)d_\epsilon\Big|^2(t)>\Lambda\Big\}.$$
Then by the weak $L^1$-estimate, we have
\begin{equation}\label{bad_set1}
\Big|B_\Lambda^T\Big|\le \frac{C_0}{\Lambda}.
\end{equation}
Now we have

\noindent{\bf Claim 8.1}. For any $t\in G_\Lambda^T$,  set $\tau_\epsilon(t):=\big(\Delta d_\epsilon+\frac{1}{\epsilon^2}(1-|d_\epsilon|^2)d_\epsilon\big)(t)$. Then there exists $\tau(t)\in L^2(\Omega,\mathbb R^3)$
such that, after passing to possible subsequences, $\tau_\epsilon(t)\rightharpoonup \tau(t)$ in $L^2(\Omega)$, and
\begin{equation}\label{slice_strong1}
d_\epsilon(t)\rightarrow d(t) \ {\rm{in}}\ H^1_{\rm{loc}}(\Omega),
e_\epsilon(d_\epsilon(t))\,dx\rightharpoonup \frac12|\nabla d(t)|^2\,dx
\ {\rm{as\ convergence\ of\ Radon\ measures\ in}}\ \Omega.
\end{equation}
In particular, $d(t)\in H^1(\Omega,\mathbb S^2_+)$ is a suitable approximated harmonic map,
with $L^2$-tension field $\tau(t)$.

\smallskip
Since $t\in G_\Lambda^T$, it follows from the definition that $\displaystyle\liminf_{\epsilon\rightarrow 0}\big\|\tau_\epsilon(t)\big\|_{L^2(\Omega)}\le\Lambda$ and hence  there exists $\tau(t)\in
L^2(\Omega,\mathbb R^3)$ such that, after passing to possible subsequences, $\tau_\epsilon(t)\rightharpoonup \tau(t)$ in $L^2(\Omega)$.
To show (\ref{slice_strong1}),  recall that by (\ref{global_energy_bound}) we can assume, after passing to possible subsequences,
$d_\epsilon(t)\rightharpoonup d(t)$ in $H^1(\Omega)$ and there exists a nonnegative
Radon measure $\nu_t$ in $\Omega$ such that
$$e_\epsilon(d_\epsilon(t))\,dx\rightharpoonup \frac12|\nabla d(t)|^2\,dx+\nu_t$$
as convergence of Radon measures in $\Omega$.  It is easy to check from the definition of $G_\Lambda^T$
that for any $t\in G_\Lambda^T$,
$\{d_\epsilon(t)\}\subset {\bf X}(C_0, \Lambda, a; \Omega)$ with $a=1$. Hence Theorem
\ref{precomp1} implies that
\begin{equation}\label{slice-conv}
\nu_t\equiv 0, \  \  d_\epsilon(t)\rightarrow d(t) \ {\rm{in}}\ H^1_{\rm{loc}}(\Omega),
\ {\rm{and}}\  \  \frac{1}{\epsilon^2}\big(1-|d_\epsilon(t)|^2\big)^2
\rightarrow 0 \ {\rm{in}}\ L^1_{\rm{loc}}(\Omega).
\end{equation}
Hence $d(t)\in H^1(\Omega,\mathbb S^2_+)$ is
an approximated harmonic map with tension field $\tau(t)\in L^2(\Omega,\mathbb R^3)$.
To see $d(t)$ is a suitable approximated harmonic map, observe that the same calculations
as in Lemma 3.1 apply to $Y\cdot\nabla d_\epsilon(t)$ for any $Y\in C^\infty_0(\Omega,\mathbb R^3)$. Hence we obtain
\begin{equation}\label{e-stationary}
\int_\Omega \Big(\big\langle \frac{\partial d_\epsilon(t)}{\partial x_i},
\frac{\partial d_\epsilon(t)}{\partial x_j}\big\rangle \frac{\partial Y^i}{\partial x_j}
-\big(\frac12|\nabla d_\epsilon(t)|^2+\frac{1}{4\epsilon^2}(1-|d_\epsilon(t)|^2)^2\big)
{\rm{div}}Y+\big\langle \tau_\epsilon(t), Y\cdot\nabla d_\epsilon(t)\big\rangle\Big)=0.
\end{equation}
After sending $\epsilon\to 0$, (\ref{e-stationary}), combined with (\ref{slice-conv}),
implies that $d(t)$ satisfies the identity (\ref{suitable2}) and hence $d(t)$ is a suitable
approximated harmonic map.  The Claim 8.1 is proven.

\smallskip
\noindent{\bf Claim 8.2}.  For any subdomain $\widetilde\Omega\subset\subset\Omega$,
it holds that
\begin{equation}\label{good-conv}
\lim_{\epsilon\rightarrow 0}
\int_{\widetilde\Omega\times G_\Lambda^T}
|\nabla(d_\epsilon-d)|^2
=0.
\end{equation}
We prove (\ref{good-conv}) by contradiction. Suppose (\ref{good-conv}) were false.
Then there exist a subdomain $\widetilde\Omega\subset\subset\Omega$, $\delta_0>0$, and $\epsilon_i\rightarrow 0$ such that
\begin{equation}\label{gap1}\int_{\widetilde\Omega\times G_\Lambda^T}
|\nabla(d_{\epsilon_i}-d)|^2\ge \delta_0.
\end{equation}
Note that from (\ref{weak_strong}) we have
\begin{equation}
\label{no-gap}
\lim_{\epsilon_i\rightarrow 0}\int_{\Omega\times G_\Lambda^T}
|d_{\epsilon_i}-d|^2=0.
\end{equation}
By Fubini's theorem, (\ref{gap1}), and (\ref{no-gap}), we have that there exists
$t_i\in G_\Lambda^T$ such that
\begin{equation}\label{no-gap1}
\lim_{\epsilon_i\rightarrow 0}\int_{\Omega}|d_{\epsilon_i}(t_i)
-d(t_i)|^2=0,
\end{equation}
and
\begin{equation}\label{gap2}
\int_{\widetilde\Omega}\big|\nabla(d_{\epsilon_i}(t_i)
-d(t_i))\big|^2\ge\frac{2\delta_0}{T}.
\end{equation}
It is easy to see that $\big\{d_{\epsilon_i}(t_i)\big\}\subset {\bf X}(C_0, \Lambda, 1; \Omega)$
and $\big\{d(t_i)\big\}\subset {\bf Y}(C_0, \Lambda, 1; \Omega)$. It follows
from Theorem 6.1 and Theorem 7.1 that there exist $d_1, d_2\in {\bf Y}(C_0, \Lambda, 1; \Omega)$ such that
$$d_{\epsilon_i}(t_i)\rightarrow d_1
\ {\rm{and}}\ d(t_i)\rightarrow d_2\ \ {\rm{in}}\ \ L^2(\Omega)\cap H^1(\widetilde\Omega).
$$
This and (\ref{gap2}) imply that
\begin{equation}\label{gap3}
\int_{\widetilde\Omega}\big|\nabla(d_1
-d_2)\big|^2\ge\frac{2\delta_0}{T}.
\end{equation}
On the other hand, from (\ref{no-gap1}), we have that
\begin{equation}\label{no-gap2}
\int_{\Omega}|d_1
-d_2|^2=0.
\end{equation}
It is clear that (\ref{gap3}) contradicts (\ref{no-gap2}). Hence the Claim 8.2 is proven.

Now we can use argument as in Lemma \ref{small_holder} Claim 4.4 to conclude that
\begin{equation}\label{xt-strong3}
\int_{\widetilde\Omega\times G_\Lambda^T}
\frac{1}{\epsilon^2}(1-|d_\epsilon|^2)^2\rightarrow 0.
\end{equation}
Combining (\ref{good-conv}) and (\ref{xt-strong3}) yields
\begin{equation}\label{xt-strong4}
\Big\|d_\epsilon-d\Big\|_{L^2_tH^1_x\big(\widetilde\Omega\times G_\Lambda^T\big)}^2
+\int_{\widetilde\Omega\times G_\Lambda^T}\frac{1}{\epsilon^2}(1-|d_\epsilon|^2)^2\rightarrow 0.
\end{equation}
On the other hand, it follows from (\ref{global_energy_bound}) and (\ref{bad_set1}) that
\begin{equation}\label{xt-strong5}
\Big\|d_\epsilon-d\Big\|_{L^2H^1_x\big(\Omega\times B_\Lambda^T\big)}^2
+\int_{\widetilde\Omega\times B_\Lambda^T}\frac{1}{\epsilon^2}(1-|d_\epsilon|^2)^2
\le \Big(2\sup_{t>0}\int_\Omega e_\epsilon(d_\epsilon)(t)\Big)\Big|B_\Lambda^T\Big|\le C\Lambda^{-1}.
\end{equation}
Hence we have
\begin{equation}\label{xt-strong6}
\lim_{\epsilon\rightarrow 0}\Big[\big\|d_\epsilon-d\big\|_{L^2H^1_x(\widetilde\Omega\times [0,T])}^2
+\int_{\widetilde\Omega\times [0,T]}\frac{1}{\epsilon^2}(1-|d_\epsilon|^2)^2\Big]
\le C\Lambda^{-1}.
\end{equation}
Since $\Lambda>1$ can be chosen arbitrarily large, this implies that $\nabla d_\epsilon\rightarrow \nabla d$ strongly
in $L^2_{\rm{loc}}(\Omega\times [0,T])$ and $\displaystyle\frac{1}{\epsilon^2}(1-|d_\epsilon|^2)^2\rightarrow 0$
in $L^1_{\rm{loc}}(\Omega\times [0,T])$.

Since $(u_\epsilon, d_\epsilon)$ solves the equation (\ref{MLLF}) along with (\ref{IBC}), it is standard
that by utilizing (\ref{xt-strong6}) and (\ref{weak_strong}) we can show that $(u,d)$ is a weak solution of the equation
(\ref{LLF}) and (\ref{IBC}). The global energy inequality (\ref{global_energy_ineq}) for $(u,d)$ follows
from (\ref{global_energy_ineq1}), with $t_1=0$ and $t=t_2>0$, by sending $\epsilon$ to $0$, with the help
of the lower semicontinuity and the observation that
$$\Delta d_\epsilon+\frac{1}{\epsilon^2}(1-|d_\epsilon|^2)d_\epsilon
=\partial_t d_\epsilon+u_\epsilon\cdot\nabla d_\epsilon
\rightharpoonup \partial_t d+u\cdot\nabla d \ \ {\rm{in}}\ \ L^2(\Omega\times [0,T]).
$$
This completes the proof of Theorem \ref{existence}. \hfill\qed

\bigskip
\noindent{\bf Proof of Theorem \ref{compactness0}}. The proof is similar to that of Theorem \ref{existence}. Here we only sketch it.
First, it follows from the equation (\ref{LLF}) and the condition (\ref{energy_bound}) that there exists $p>2$ such that for
$0<T<+\infty$,
\begin{equation}\label{t-uniform1}
\sup_{k}\Big[\|\partial_tu_k\|_{L^\frac54(\Omega\times [0,T])+L^2([0,T], H^{-1}(\Omega))+L^2([0,T], W^{-1,p}(\Omega))}
+\|\partial_t d_k\|_{L^\frac54(\Omega\times [0,T])}\Big]<+\infty.
\end{equation}
Hence, by Aubin-Lions' Lemma \cite{temam} we have that after taking to possible subsequences,  there exists $u\in L^\infty_tL^2_x\cap L^2_tH^1_x(\Omega\times
[0,T],\mathbb R^3)$ and $d\in L^\infty_tH^1_x(\Omega\times [0,T], \mathbb S^2)$ such that
\begin{equation}\label{weak_strong1}
(u_k, d_k)\rightarrow (u, d) \ {\rm{in}}\ L^2_{\rm{loc}}(\Omega\times [0,T]),
\ (\nabla u_k,\nabla d_k)\rightharpoonup (\nabla u, \nabla d) \ {\rm{in}}\ L^2_tL^2_x(\Omega\times [0,T]).
\end{equation}
It follows from (\ref{energy_bound}) and Fatou's lemma that
\begin{equation}\label{fatou}
\int_0^{\infty}\liminf_{k\rightarrow \infty}\int_\Omega
\big|\Delta d_k+|\nabla d_k|^2d_k\big|^2
\le C_0.
\end{equation}
For $\Lambda>>1$, define
$$G_\Lambda^T:=\Big\{t\in [0,T]:
\ \liminf_{k\rightarrow \infty}\int_\Omega
\big|\Delta d_k+|\nabla d_k|^2d_k\big|^2(t)\le \Lambda\Big\},$$
and
$$B_\Lambda^T:=[0,T]\setminus G_\Lambda^T=\Big\{t\in [0,T]:
\ \liminf_{k\rightarrow \infty}\int_\Omega
\big|\Delta d_k+|\nabla d_k|^2d_k\big|^2(t)>\Lambda\Big\}.$$
Then by the weak $L^1$-estimate, we have
\begin{equation}\label{bad_set}
\Big|B_\Lambda^T\Big|\le \frac{C_0}{\Lambda}.
\end{equation}
Now we have

\noindent{\bf Claim 8.3}. For $L^1$ a.e. $t\in G_\Lambda^T$, set $\tau_k(t)=(\Delta d_k+|\nabla d_k|^2 d_k)(t)$.
Then there exists $\tau(t)\in L^2(\Omega,\mathbb R^3)$ such that, after passing to subsequences,
$\tau_k(t)\rightharpoonup \tau(t)$ in $L^2(\Omega)$, and
\begin{equation}\label{slice_strong}
d_k(t)\rightarrow d(t) \ {\rm{in}}\ H^1_{\rm{loc}}(\Omega).
\end{equation}
In particular, $d(t)\in H^1(\Omega, \mathbb S^2_{-1+a})$ is a suitable approximated harmonic map,
with tension field $\tau(t)$.

\smallskip
Since $t\in G_\Lambda^T$, we have $\displaystyle\liminf_{k\rightarrow\infty}\big\|\tau_k(t)\big\|_{L^2(\Omega)}\le\Lambda$
and hence there exists $\tau(t)\in L^2(\Omega,\mathbb R^3)$ such that, after taking a subsequence, $\tau_k(t)\rightharpoonup \tau(t)$ in $L^2(\Omega)$. To show (\ref{slice_strong}),  recall that by (\ref{energy_bound}) we can assume, after passing to subsequences,
$d_k(t)\rightharpoonup d(t)$ in $H^1(\Omega)$ and there exists a nonnegative
Radon measure $\nu_t$ in $\Omega$ such that
$$\frac12|\nabla d_k|^2(t)\,dx\rightharpoonup \frac12|\nabla d(t)|^2\,dx+\nu_t$$
as convergence of Radon measures in $\Omega$.  It is easy to check from the definition
that for $L^1$ a.e. $t\in G_\Lambda^T$, $\{d_k(t)\}$ is a family of suitable approximated harmonic map
such that  $\{d_k(t)\}\subset {\bf Y}(C_0, \Lambda, a; \Omega)$ with $0<a\le 2$. Hence Theorem
\ref{precomp2} implies that $\nu_t\equiv 0$ and $d_k(t)\rightarrow d(t)$
strongly in $H^1_{\rm{loc}}(\Omega)$. Hence $d(t)\in H^1(\Omega,\mathbb S^2_{-1+a})$ is a suitable approximated harmonic map.
This proves the Claim 8.3.

\smallskip
\noindent{\bf Claim 8.4}. For any subdomain $\widetilde\Omega\subset\subset\Omega$, it holds
that
\begin{equation}\label{good-conv1}
\lim_{k\rightarrow \infty}\int_{\widetilde\Omega}|\nabla (d_k-d)|^2=0.
\end{equation}
Similar to the proof of Claim 8.2, (\ref{good-conv1}) can be proven by contradiction.
For, otherwise, there exist $\widetilde\Omega\subset\subset\Omega$, $\delta_0>0$ and $k_l\rightarrow\infty$ such that
$$\int_{\Omega\times G_\Lambda^T}|d_{k_l}-d|^2\rightarrow 0,
\ {\rm{and}}\ \int_{\widetilde\Omega\times G_\Lambda^T}|\nabla (d_{k_l}-d)|^2\ge \delta_0.
$$
By Fubini's theorem, there exists $\{t_l\}\subset G_\Lambda^T$ such that
\begin{equation}\label{slice-conv1}
\int_{\Omega}|d_{k_l}(t_l)-d(t_l)|^2\rightarrow 0,
\ {\rm{and}}\ \int_{\widetilde\Omega}|\nabla (d_{k_l}(t_l)-d(t_l))|^2\ge \frac{2\delta_0}{T}.
\end{equation}
Since $\{d_{k_l}(t_l)\}, \{d(t_l)\}\subset {\bf Y}(C_0, \Lambda, a;\Omega)$ for $0<a\le 2$,
it follows from Theorem 7.1 that there exist $d_1, d_2\in {\bf Y}(C_0, \Lambda, a;\Omega)$
such that
$$d_{k_l}(t_l)\rightarrow d_1, \ \ d(t_l)\rightarrow d_2 \ \ {\rm{in}}\ \ L^2(\Omega)\cap H^1_{\rm{loc}}(\Omega).
$$
Hence, by (\ref{slice-conv1}), we obtain
\[
\int_{\Omega}|d_1-d_2|^2=0,
\ {\rm{and}}\ \int_{\widetilde\Omega}|\nabla (d_1-d_2|^2\ge \frac{2\delta_0}{T}.
\]
This is impossible. Thus we obtain
\begin{equation}\label{xt-strong7}
\Big\|d_k-d\Big\|_{L^2_tH^1_x\big(\widetilde\Omega\times G_\Lambda^T\big)}\rightarrow 0.
\end{equation}
On the other hand, it follows from (\ref{energy_bound}) and (\ref{bad_set}) that
\begin{equation}\label{xt-strong8}
\Big\|d_\epsilon-d\Big\|_{L^2H^1_x\big(\widetilde\Omega\times B_\Lambda^T\big)}^2
\le \Big(\sup_{t>0}\int_\Omega |\nabla d_k|^2(t)\Big)\Big|B_\Lambda^T\Big|\le C\Lambda^{-1}.
\end{equation}
Hence we have
\begin{equation}\label{xt-strong9}
\lim_{k\rightarrow \infty}\big\|d_k-d\big\|_{L^2H^1_x(\widetilde\Omega\times [0,T])}^2
\le C\Lambda^{-1}.
\end{equation}
Since $\Lambda>1$ can be chosen arbitrarily large, this implies that $d_k\rightarrow d$ strongly
in $H^1_{\rm{loc}}(\Omega\times [0,T])$. Since $(u_k, d_k)$ solves the equation (\ref{LLF}), it is standard
that by utilizing (\ref{xt-strong9}) and (\ref{weak_strong1}) we can show that $(u,d)$ is a weak solution of the equation
(\ref{LLF}). \hfill\qed

\bigskip
\noindent{\bf Acknowledgements}.  The first author is partially supported by NSF grants DMS1065964 and DMS1159313. The
second author is partially supported by NSF grants DMS1001115 and DMS 1265574 and NSFC grant 11128102.

\end{document}